\documentclass[11pt]{amsart}

\usepackage[dvipsnames]{xcolor}
\usepackage{amssymb,verbatim}
\usepackage{amsmath,amsfonts,enumitem}
\usepackage[mathscr]{euscript} 
\usepackage{amsthm}
\usepackage{url}
\usepackage{graphicx} 
\usepackage{enumitem} 
\usepackage{hyperref} 
\hypersetup{colorlinks} 
\usepackage{bbm} 
\usepackage{bm} 
\usepackage{mathrsfs} 
\usepackage{cancel} 

\usepackage[colorinlistoftodos]{todonotes}

\RequirePackage{cleveref}
\usepackage{hypcap}
\hypersetup{colorlinks=true, citecolor=darkblue, linkcolor=darkblue}
\definecolor{darkblue}{rgb}{0.0,0,0.7}
\newcommand{\darkblue}{\color{darkblue}}

\definecolor{darkred}{rgb}{0.68,0,0}
\newcommand{\darkred}{\color{darkred}}

\definecolor{darkgreen}{rgb}{0,.38,0}
\newcommand{\darkgreen}{\color{darkgreen}}

\newcommand{\defn}[1]{\emph{\darkblue #1}}
\newcommand{\defna}[1]{\emph{\darkred #1}}

\newcommand{\defng}[1]{\emph{\darkgreen #1}}

\setlist[enumerate]{
	label=\textnormal{({\roman*})},
	ref={\roman*}}

\makeatletter
\def\th@plain{%
	\thm@notefont{}
	\itshape 
}
\def\th@definition{%
	\thm@notefont{}
	\normalfont 
}
\makeatother

\newtheorem{thm}{Theorem}[section]

\newtheorem*{claim*}{Claim}
\newtheorem{cor}[thm]{Corollary}
\newtheorem{Def}[thm]{Definition}

\newtheorem{prop}[thm]{Proposition}
\newtheorem{conj}[thm]{Conjecture}
\newtheorem{conv}[thm]{Convention}
\newtheorem{question}[thm]{Question}

\theoremstyle{definition}
\newtheorem{ex}[thm]{Example}
\newtheorem{rem}[thm]{Remark}

\numberwithin{figure}{section}
\numberwithin{equation}{section}

\def\emp{\nothing}

\def\bd{\blacktriangledown}

\def\zz{\mathbb Z}
\def\nn{\mathbb N}

\def\rr{\mathbb R}
\def\qqq{\mathbb Q}

\def\ov{\overline}

\def\sm{\smallsetminus}

\def\Ga{\Gamma}

\def\la{\lambda}
\def\ga{\gamma}
\def\si{\sigma}
\def\de{\delta}

\def\al{\alpha}
\def\be{\beta}
\def\om{\omega}

\def\cD{\mathcal D}

\def\cM{\mathcal M}

\def\cT{\mathcal T}

\def\cX{\mathcal X}
\def\cY{\mathcal Y}

\def\CT{\text{CT}}
\def\CTr{\text{\em CT}}

\def\ssu{\subset}

\def\<{\langle}
\def\>{\rangle}

\def\rH{ {\text {\rm H} } }

\def\0{{\mathbf 0}}

\def\st{{\rm st}}
\def\ma{{\rm ma}}

\def\nothing{\varnothing}

\def\.{\hskip.06cm}
\def\ts{\hskip.03cm}

\def\ba{\textbf{\textbf{a}}}
\def\bb{\textbf{\textbf{b}}}
\def\bd{\textbf{\textbf{d}}}

\def\rba{\textbf{\em \textbf{a}}}
\def\rbb{\textbf{\em \textbf{b}}}

\def\PerC{C$_{\text{\sc PER}}$}
\def\PerV{V$_{\text{\sc PER}}$}

\newcommand{\red}{\mathrm{red}}
\newcommand{\pc}{\mathrm{PC}}
\newcommand{\PM}{\mathrm{PM}}
\newcommand{\per}{\mathrm{per}}

\newcommand{\SYT}{\operatorname{{\rm SYT}}}

\def\.{\hskip.06cm}
\def\ts{\hskip.03cm}

\def\nin{\noindent}

\newcommand{\textsu}[1]{\textup{\textsf{#1}}}

\newcommand{\ComCla}[1]{\textup{\textsu{#1}}}

\newcommand{\sharpP}{\ComCla{\#P}}
\newcommand{\SP}{\ComCla{\#P}}

\newcommand{\GapP}{\ComCla{GapP}}

\newcommand{\Sigmap}{\ensuremath{\Sigma^{{\textup{p}}}}}

\newcommand{\NP}{\ComCla{NP}}
\newcommand{\BPP}{\ComCla{BPP}}
\newcommand{\coNP}{\ComCla{coNP}}
\renewcommand{\P}{\ComCla{P}}
\newcommand{\CeqP}{\ComCla{C$_=$P}}

\newcommand{\PH}{\ComCla{PH}}
\newcommand{\PSPACE}{\ComCla{PSPACE}}
\newcommand{\FP}{\ComCla{FP}}
\newcommand{\PP}{\ComCla{PP}}

\def\SP{\sharpP}

\def\CEP{{\CeqP}}

\newcommand{\inv}{\operatorname{{\rm inv}}}

\title[Computational complexity of counting coincidences]
{Computational complexity of counting coincidences}
\date{\today}

\author{Swee Hong Chan}
\address[Swee Hong Chan]{Department of Mathematics, Rutgers University,  Piscatway, NJ 08854.}
\email{\texttt{sc2518@rutgers.edu}}

\author[\ts Igor Pak]{Igor Pak}
\address[Igor Pak]{Department of Mathematics, UCLA,  Los Angeles, CA 90095.}
\email{\texttt{pak@math.ucla.edu}}


\begin{document}

\begin{abstract}
Can you decide if there is a coincidence in the numbers counting
two different combinatorial objects?  For example, can you decide
if two regions in $\rr^3$ have the same number of domino tilings?
There are two versions of the problem, with \ts $2\times 1 \times 1$
\ts and \ts $2\times 2 \times 1$ \ts boxes.  We prove that in both
cases the coincidence problem is not in the polynomial hierarchy
unless the polynomial hierarchy collapses to a finite level.
While the conclusions are the same, the proofs are notably
different and generalize in different directions.

We proceed to explore the coincidence problem for counting independent
sets and matchings in graphs, matroid bases, order ideals and linear
extensions in posets, permutation patterns, and the Kronecker coefficients.
We also make a number of conjectures for counting other combinatorial objects
such as plane triangulations, contingency tables, standard Young tableaux,
reduced factorizations and the Littlewood--Richardson coefficients.
\end{abstract}
	
\maketitle
	
\vskip.2cm


\section{Introduction}\label{s:intro}

\subsection{Tilings} \label{ss:intro-tilings}
In this paper we consider coincidences of combinatorial counting functions.
Consider two bounded regions in the plane.  Do they have
the same number of domino tilings?  Here we are assuming that the
regions are finite subsets of unit squares on a square grid, we write
\ts $\Ga \ssu \zz^2$, and the dominos are the usual \ts $2\times 1$ \ts
rectangles. For example, a \ts $2\times 3$ \ts rectangle has three
domino tilings, and both regions on the right have four
domino tilings:

\begin{figure}[hbt]
\begin{center}
	\includegraphics[height=1.5cm]{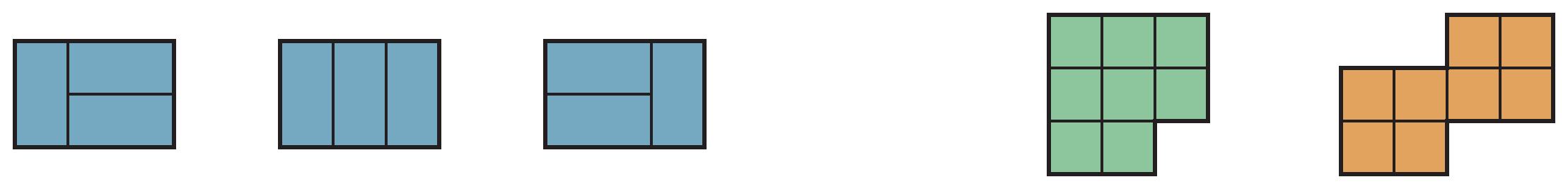}
\end{center}
\label{f:basic}
\end{figure}


Algorithmically, the coincidence problem is easy, since we can simply
compute the number of domino tilings of each region in polynomial time,
and then compare the numbers.  Indeed, the \emph{Kasteleyn formula} \ts
gives the number \ts $\tau(\Ga)$ \ts of domino tilings as an \ts $n\times n$
determinant which can be computed in time polynomial in~$n$, where \ts
$n=|\Ga|$ \ts is the area of region~$\Ga$, see e.g.\ \cite{Ken,LP}.

Now consider what happens to domino tilings in $\rr^3$.  Should we
expect that the coincidence problem remains easy?  It turns out,
this is a much harder problem.  In fact, there
are two versions of $3$-dimensional dominoes: a \ts $2\times 2 \times 1$ \ts box
which we call a \defn{slab}, and a \ts $2\times 1 \times 1$ \ts
box which we call a \defn{brick}.  For example, there are three tilings of a
\ts $2\times 2 \times 2$ \ts box with a slab and nine with a brick,
shown below up to rotations:


\begin{figure}[hbt]
\begin{center}
	\includegraphics[height=1.4cm]{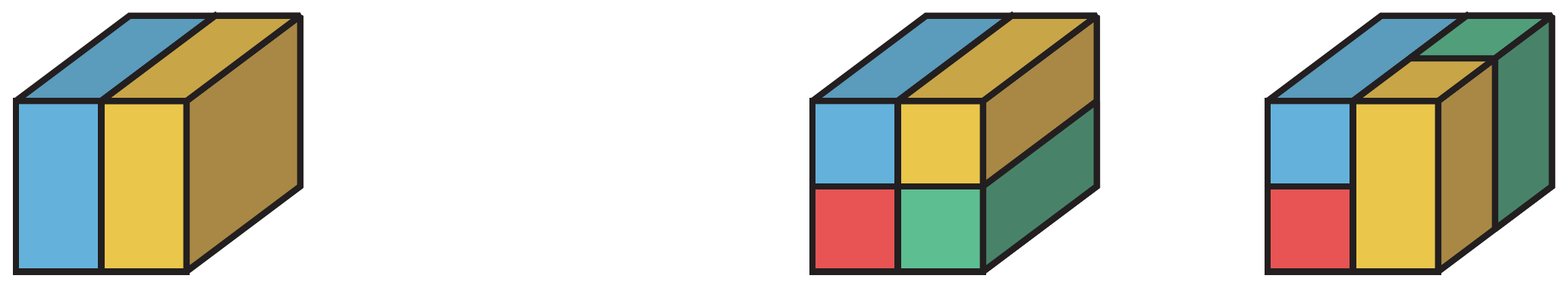}
\end{center}
\label{f:2box}
\end{figure}

For both versions, there is no analogue of the Kasteleyn formula for
the number of tilings.  Indeed, for tilings with slabs or
with bricks, Jed Yang and the second author proved that computing the
number of tilings is $\SP$-complete \cite{PY}.  But this is where
the similarities end: the coincidence problems have very different
nature in these two cases.  The reason for the difference is the type
of reduction used in the proofs (see Section~\ref{s:perm} and~$\S$\ref{ss:finrem-two-types}).
\smallskip

Let \ts $\tau_s(\Ga)$ \ts be the number of tilings of a region \ts
$\Ga \ssu \zz^3$ \ts with slabs.  Denote by \ts {\sc C$_{\text{ST}}$} \ts the
\defn{slab tiling coincidence problem}: \.
\begin{equation}\label{eq:TST}
\text{\sc C$_{\text{ST}}$} \ := \ \big\{\. \tau_s(\Ga) \ts =^? \ts \tau_s(\Ga'),
\ \ \text{where} \ \ \Ga,\Ga'\ssu \zz^3\big\}.
\end{equation}

\smallskip

\begin{thm}\label{t:slab-main}
Problem \ts {\sc C$_{\text{ST}}$} \ts is \ts $\coNP$-hard.
Furthermore, \ts {\sc C$_{\text{ST}}$} \ts is
not in the polynomial hierarchy unless the polynomial hierarchy
collapses to a finite level:  \.  \ts {\sc C$_{\text{ST}}$} \ts
$\in \PH \ \Rightarrow \ \PH=\Sigmap_m$ \. for some $m$.
\end{thm}

\smallskip

In particular, the theorem implies that \ts {\sc C$_{\text{ST}}$} \ts does
not have a polynomial time algorithm (unless $\P=\NP$).  Nor does it have
a probabilistic polynomial time algorithm (unless $\PH$ collapses),
since \ts $\BPP \subseteq \PH$.  Nor does there exist a polynomial
size witness for the coincidence (ditto).

The first part of the theorem follows immediately from
$\NP$-completeness of the slab tileability \cite[Thm~1.1]{PY},
which is a special cases of the \ts {\sc C$_{\text{ST}}$}.
The second part follows from the parsimonious reduction
in the proof of $\SP$-completeness of the \ts {\sc C$_{\text{ST}}$},
combined with some known computational complexity (see $\S$\ref{ss:perm-slab}).

\smallskip

Let \ts $\tau_b(\Ga)$ \ts be the number of tilings of a region \ts
$\Ga \ssu \zz^3$ \ts with bricks.  We similarly denote by \ts {\sc C$_{\text{BT}}$} \ts the
\defn{brick tiling coincidence problem}: \.
\begin{equation}\label{eq:BST}
\text{\sc C$_{\text{BT}}$} \ := \ \big\{\. \tau_b(\Ga) \ts =^? \ts \tau_b(\Ga'),
\ \ \text{where} \ \ \Ga,\Ga'\ssu \zz^3\big\}.
\end{equation}

\smallskip

\begin{thm}\label{t:brick-main}
Problem \ts {\sc C$_{\text{BT}}$} \ts is not in the polynomial hierarchy unless the polynomial hierarchy
collapses to a finite level:  \.  \ts {\sc C$_{\text{BT}}$} \ts
$\in \PH \ \Rightarrow \ \PH=\Sigmap_m$ \. for some $m$.
\end{thm}

\smallskip

Note that we make no claims that \ts {\sc C$_{\text{BT}}$} \ts is \ts $\coNP$-hard.
Proof of that would require a major advance in computational complexity.
However, the theorem does imply that \ts {\sc C$_{\text{BT}}$} \ts is not in \ts $\P$,
\ts $\BPP$, \ts $\NP$, \ts $\coNP$, etc., unless \ts $\PH=\Sigmap_m$.
For the proof we again need some standard results in computational
complexity, combined with a curious combinatorial result of
independent interest (Theorem~\ref{t:domino}).

\smallskip

\subsection{Beyond tilings} \label{ss:intro-gen}
The paper starts with proofs of two theorems above which allows us to
develop tools to prove similar results for the coincidence problem of many
other combinatorial counting functions.  We group the results into
two, by analogy with the two theorems above.

Let \ts $f\in \SP$ \ts be a counting function.  The \defn{coincidence problem} \ts {\sc C}$_f$ \ts
is defined as follows:
$$\text{\sc C}_f \ := \ \big\{\ts f(x) \. = ^? \ts f(y) \ts \big\}.
$$

\begin{thm}\label{t:main-first}
The coincidence problem \. {\sc C}$_f$ \. is not in the polynomial hierarchy
unless the polynomial hierarchy  collapses to a finite level,
where \ts $f$ \ts
is either one of the following:

\smallskip

{\small $(0)$} \ the number of satisfying assignments of a \ts {\sc 3SAT} \ts formula,

{\small $(1)$} \ the number of proper $3$-colorings of a planar graph,

{\small $(2)$} \ the number of Hamiltonian cycles in a graph,

{\small $(3)$} \ the number \ts $\pc_\pi(\si)$ \ts of patterns~$\ts \pi \ts$ in a permutation~$\ts \si$,

{\small $(4)$} \ the Kronecker coefficient \ts $g(\la,\mu,\nu)$ \ts for partitions \ts
$\la,\mu,\nu\vdash n$ \ts given in unary.

\smallskip

\nin
Furthermore, the coincidence problem \. {\sc C}$_f$ \. is \ts $\coNP$-hard in all these cases.
\end{thm}

\smallskip

As we explain, the theorem follows from known complexity results.
This is in sharp contrast with our next theorem where each
part requires additional work.  We need just one definition.

\smallskip

A \defn{rational matroid} \ts is a matroid that is realizable over~$\qqq$.
We assume that this matroid is given by a set of $n$ vectors in~$\qqq^d$.
In this presentation, computing the number \ts $b(M)$ \ts of bases of a rational
matroid~$M$ is known to be $\SP$-complete \cite{Snook} (cf.~$\S$\ref{ss:finrem-matroids}).

\smallskip

\begin{thm}[{\rm main theorem}{}]
\label{t:main-combined}
The coincidence problem \. {\sc C}$_f$ \. is not in the polynomial hierarchy
unless the polynomial hierarchy  collapses to a finite level,
where \ts $f$ \ts is either of the following:

\smallskip

{\small $(0)$} \ the number of perfect matchings \ts $\PM(P)$ \ts in a simple bipartite graph,

{\small $(1)$} \ the number of satisfying assignments of a \ts {\sc MONOTONE 2SAT} \ts formula,

{\small $(2)$} \ the number of independent sets \ts $\la(G)$ \ts in a planar bipartite graph,

{\small $(3)$} \ the number of order ideals \ts $\mu(P)$ \ts of a poset,

{\small $(4)$} \ the number of linear extensions \ts $e(P)$ \ts of a $2$-dimensional poset,

{\small $(5)$} \ the number of matchings \ts $\ma(G)$ \ts in a planar bipartite graph,

{\small $(6)$} \ the number of bases \ts $b(M)$ \ts of a rational matroid.
\end{thm}

\smallskip

Note that all counting functions in Theorems~\ref{t:main-first} and~\ref{t:main-combined}
are \ts $\SP$-complete.  Unfortunately, that by itself does not imply that the
corresponding coincidence problems are not in~$\PH$ (see~$\S$\ref{ss:finrem-coincidence}).
Example~\ref{ex:LE-restricted} (see also~$\S$\ref{ss:finrem-LE-height-two})
has an especially notable variation on part~{\small $(4)$} in the theorem for posets of height two.
While the number of linear extensions is known to be \ts $\SP$-complete in this case,
it is open whether the corresponding coincidence problem is in~$\PH$.  Similarly,
a variation on part~{\small $(6)$} for bicircular matroids gives another example of
this type (see~$\S$\ref{ss:finrem-matroids}).  The number of bases is known to
be \ts $\SP$-complete in this case, but corresponding coincidence problem
remains unexplored.

\smallskip


\subsection*{Paper structure}  We start with basic definitions and notation
in a short Section~\ref{s:def}. In  Section~\ref{s:domino}, we give some
preliminary results on domino tilings. In Section~\ref{s:perm}, we prove
Theorems~\ref{t:slab-main}, \ref{t:brick-main} and~\ref{t:main-first}.
In a lengthy Section~\ref{s:var}, we discuss further examples,
and prove Main Theorem~\ref{t:main-combined}.
We conclude with many final remarks and open problems in Section~\ref{s:finrem}.

\medskip

\section{Definitions and notation}\label{s:def}

\smallskip

\subsection*{General notation}
Let \ts $[n]=\{1,\ldots,n\}$ \ts and \ts $\nn=\{0,1,2,\ldots\}$.
For a sequence \ts $\ba =(a_1,\ldots,a_m)$, denote
\ts $|a| := a_1 + \ldots + a_m$\ts.  Similarly, for the integer
partitions \ts $\mu \subset \la$,  let the \defn{size} \ts $|\la/\mu|$ \ts be
the number of squares in the skew Young diagram~$\la/\mu$.  For \ts $|\la|=n$
\ts we also write \ts $\la\vdash n$.

\subsection*{Combinatorics}
We think of \ts $\zz^d \ssu \rr^d$ \ts both as a lattice and a collection
of the corresponding unit $d$-cube.  A \defn{region} \ts $\Ga \ssu \zz^d$ \ts
is a subset of $d$-cubes.  Denote by \ts $|\Ga|$ \ts the size of~$\Ga$,
which can also be viewed as the volume of the union of the corresponding unit cubes.
Region \ts $\Ga$ \ts is called \defn{simply connected} if the union of the
corresponding (closed) $d$-cubes is simply connected.  A \defn{tile}
\ts in \ts $\rr^d$ \ts  is a finite simply connected region.
For a set of tiles \ts $T=\{t_1,\ldots,t_m\}$, a \defn{tiling} \ts is a disjoint
union of copies of tiles~$t_i$ (unless stated otherwise, parallel translations,
rotations and reflections are allowed).

We assume that the reader is familiar with basic notions in algebraic combinatorics,
such as \ts \emph{standard Young tableaux}, \emph{Kostka numbers},
\emph{Littlewood--Richardson} \ts and \emph{Kronecker coefficients}.
Defining them, explaining their importance, combinatorial interpretations and
properties would take too much space and be a distraction.  We refer the
reader to \cite{Mac95,Man01,Sta-EC} and further references sprinkled
throughout the paper.

\subsection*{Complexity}
We assume that the reader is familiar with basic notions and results in
computational complexity and only recall a few definitions.  We use standard
complexity classes \. $\P$, \. $\NP$,\. $\coNP$, \. $\SP$, \. $\Sigmap_m$ \. and \. $\PH$.
The notation \. $\{a =^? b\}$ \. is used to denote the
decision problem whether \ts $a=b$.  We use the \emph{oracle notation} \ts
{\sf K}$^{\text{\sf L}}$ \ts for two complexity classes \ts {\sf K}, {\sf L} $\subseteq \PSPACE$,
and the polynomial closure \ts $\<${\sc A}$\>$ for a problem \ts {\sc A} $\in \PSPACE$.
We will also use less common classes \.
$$
\GapP:= \{f-g \mid f,g\in \SP\} \quad \text{and} \quad
\CEP:=\{f(x)=^?g(y) \mid f,g\in \SP\}.
$$
The distinction between \emph{binary} \ts and \emph{unary} \ts presentation
will also be important.  We refer to \cite{GJ78} and \cite[$\S$4.2]{GJ79}
for the corresponding notions of $\NP$-completeness and \emph{strong} \ts $\NP$-completeness.

We also that assume the reader is familiar with standard decision and
counting problems, such as \ts {\sc 2SAT}, \ts {\sc MONOTONE 2SAT},
\ts {\sc 3SAT}, \ts {\sc 1-in-3 SAT},  \ts {\sc HAMILTON CYCLE},
\ts {\sc \#2SAT}, \ts {\sc \#3SAT},
\ts {\sc PERMANENT}, etc.
Occasionally, we conflate counting functions $f$ and the problems
of computing~$f$.  We hope this does not lead to a confusion.

We refer to \cite{AB,Gold,MM,Pap} for definitions and standard results
in computational complexity.  See \cite{GJ79} for the classical introduction
and a long list of $\NP$-complete problems.  See also \cite[$\S$13]{Pak-OPAC}
for a recent overview of $\SP$-complete problems in combinatorics.
For surveys on counting complexity, see \cite{For,Scho}.

\medskip


\section{Counting tilings}\label{s:domino}


\subsection{Domino tilings}\label{ss:domino-main}
Denote by \ts $\cT(n)$ \ts the set of numbers of domino tilings
over all regions of size~$2n$:
$$
\cT(n) \ := \ \big\{\. \tau(\Ga)\., \ \ \text{where} \ \ \Ga \ssu \zz^2, \ |\Ga| \ts =\ts 2n\.\big\}.
$$
Clearly, \ts $\cT(n) \subseteq \{0,1,\ldots,4^n\}$ \ts since each domino
tilings is determined by the $4$ choices for a domino at every
even square.  The following result proves a converse:

\smallskip

\begin{thm}  \label{t:domino}
There is a constant \ts $c>1$, such that \. $\cT(n) \supseteq \{0,1,\ldots,c^n\}$,
for all \ts $n\ge 1$.  Moreover, for all \ts $k \le c^n$, a region \ts
$\Ga\ssu \zz^2$ \ts with \ts $\tau(\Ga)=k$ \ts and \ts $|\Ga|=2n$,
can be constructed in time polynomial in~$n$.
\end{thm}

\smallskip

\begin{rem}\label{r:Nadeau}
{\rm
It was known before that \. $\cup_n \cT(n) = \nn$,  i.e.\ that every
nonnegative integer is the number of domino tilings of \emph{some} \ts
region.  This was shown by Philippe Nadeau with an elegant explicit
construction.\footnote{See
\href{https://mathoverflow.net/a/178100/4040}{mathoverflow.net/a/178100}\ts.}
Unfortunately, this construction has \ts $\tau(\Ga)=\Theta(n)$ \ts
where \ts $n=|\Ga|$,  and thus does not give our theorem.

In a different direction, Brualdi and Newman in~\cite{BN65} gave
(in the language of permanents), an explicit construction of
a simple bipartite graph on \ts $n$ \ts vertices with exactly~$k$
perfect matchings, for all \ts $0\le k \le 2^{n-1}$.
These graphs have unbounded degree and thus very far
from being grid graphs (or even planar graphs).  There was
a recent constant factor improvement in~\cite{GT18}, but relatively
little attention otherwise
(see \cite[\href{http://oeis.org/A089477}{A089477}]{OEIS}), compared to the
corresponding determinant problem (see~$\S$\ref{ss:finrem-det}).
}\end{rem}

\smallskip

\begin{proof}[Proof of Theorem~\ref{t:domino}]
For an integer \ts $k\ge 0$, we give an explicit construction
of a region \ts $\Ga\ssu \zz^2$ \ts with \ts $\tau(G)=k$ \ts and
\ts $|\Ga| = O(\log k)$.  This implies the result.

A square \ts $(i,j)\in \zz^2$ \ts is called \defn{even}
(\defn{odd}\.{}) if \. $i+j$ \. is even (odd). Denote by \ts
$\Ga_e=\Ga\cap \zz_e^2$ \ts and \ts $\Ga_o:=\Ga\cap \zz_o^2$ \ts the
sets of even and odd squares in~$\Ga$, respectively.
Clearly, if \ts $|\Ga_e|\ne |\Ga_o|$ \ts then \ts $\tau(\Ga)=0$.
%
%
We use \ts $\Ga-x-y$ \ts to denote
$\Ga$ with squares $x,y$ removed.
Denote by \ts $\cD(a,b)$ \ts the set of regions \ts $\Ga$ such that
\ts $\tau(\Ga)=a$ \ts and \ts $\tau(\Ga-x-y) = b$ \ts for some \ts $x\in \Ga_e$ \ts
and \ts $y\in \Ga_o$.

Below we give constructions which prove the following implications:
\begin{equation}\label{eq:imply}
\aligned
& \cD(a,1) \ne \emp \ \Longrightarrow  \ \cD(2a,1) \ne \emp\ts, \\
& \cD(a,1) \ne \emp \ \Longrightarrow  \ \cD(2a+1,1) \ne \emp\ts.
\endaligned
\end{equation}
Starting with  \. $\cD(1,1)\ne \emp$ \ts and iterating these
in \ts $O(\log k)$ \ts times gives the desired \ts $\Ga\in \cD(k,1)$.

For a region \ts $\Ga$ \ts and squares \ts $x,y\in\Ga$, we say that \ts $(\Ga,x,y)$ \ts is a \defn{$C$-triple} \ts if

\smallskip

\qquad $\circ$ \  $x=(i,j)\in \Ga_e$ \ts and \ts $y=(i,j+7)\in \Ga_o$\ts,

\qquad $\circ$ \  $(u,v)\notin\Ga$ \ts for all \ts $u>i$,

\qquad $\circ$ \  $(u,v)\notin\Ga$ \ts for all \ts $i-1\le u \le i$ \ts and \ts $j+1 \le v\le  j+6$.

\smallskip
\begin{figure}[hbt]
\begin{center}
	\includegraphics[height=3.3cm]{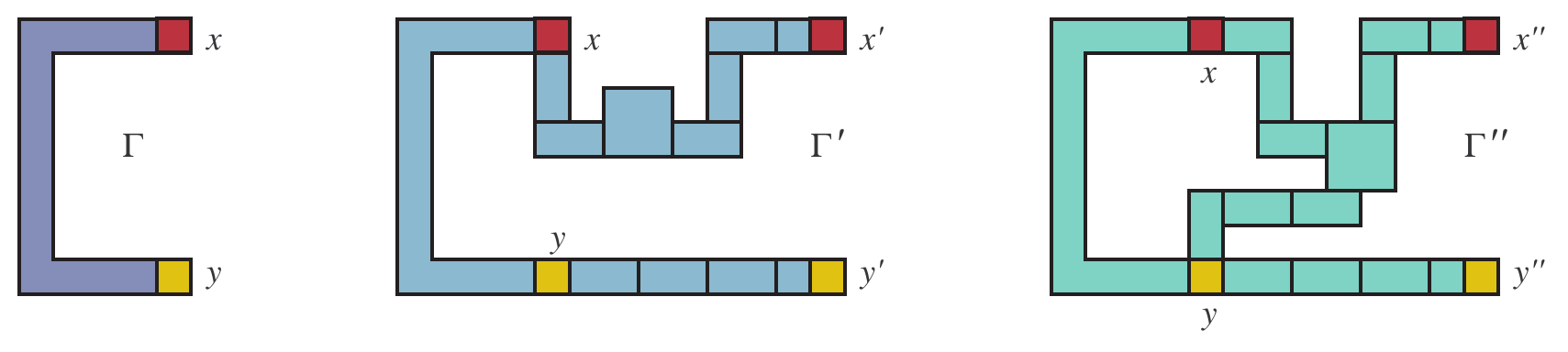}
\vskip-.3cm
\caption{Region $\Ga \in \cD(1,1)$, and two transformations \ts $\Ga\in \cD(a,1) \ts \Rightarrow \ts \Ga'\in \cD(2a,1)$,
and \ts $\Ga\in \cD(a,1) \ts \Rightarrow \ts \Ga''\in \cD(2a+1,1)$.}
\label{f:imply}
\end{center}
\end{figure}

\nin
We start with a $C$-triple \ts $(\Ga,x,y)\in \cD(1,1)$ \ts as in Figure~\ref{f:imply}.
We then define two transformations \. $(\Ga,x,y)\to (\Ga',x',y')$ \. and \.
$(\Ga,x,y)\to (\Ga'',x'',y'')$ \. as in the figure, which prove the implications~\eqref{eq:imply}.
Note that in both cases we obtain two $C$-triples by adding at most 30 squares.  This completes
the construction.  \end{proof}


\smallskip

\begin{cor}\label{c:gen}
In notation of the proof above, we have \ts $\cD(a,b)\ne \emp$ \ts
for all \ts $a,b\ge 0$.
\end{cor}

\begin{proof}
By modifying our two transformations, one can show that \.
\begin{equation}\label{eq:two}
\aligned
& \cD(a,b)\ne\emp \ \Longrightarrow \ \cD(b,a)\ne\emp\ts, \\
& \cD(a,b)\ne\emp \ts, \ \cD(a',b')\ne\emp \ \Longrightarrow \ \cD(aa'+bb',bb')\ne\emp\ts.
\endaligned
\end{equation}
The first of these is given by \ts $(\Ga,x,y) \to (\Ga',x',y')$ \ts as in Figure~\ref{f:ext}.
Similarly, the second is given by $(\Ga,x,y), (\Ga,x',y') \to (\Ga'',x,y')$ \ts as in
the figure.  Note that \ts $(\Ga'',x,y')$ \ts is no longer a $C$-triple, so this
transformation can be used only once.

In the proof above, we showed that \ts $\cD(n,1)\ne \emp$ \ts for all \ts  $n\ge 0$.
By the first transformation in~\eqref{eq:two}, this implies that \ts $\cD(1,m)\ne \emp$ \ts
for all \ts  $m\ge 0$.  Therefore, by the second transformation in~\eqref{eq:two},
we have \ts $\cD(m+n,m)\ne \emp$.  We then have \ts $\cD(m,m+n)\ne \emp$ \ts by the first
(after a modification where $x'$ and $y'$ are above and placed below $x'$ and $y'$,
respectively). Note that the second transformation is used only here, and thus
the construction is well defined.

The remaining cases in \ts $\cD(n,0)$ \ts and \ts $\cD(0,n)$, follow from
the two transformations above applied to \ts $\cD(1,0)$.   Alternatively,
they follow the last two  transformations in Figure~\ref{f:ext};
the details are straightforward.  This finishes the proof.
\end{proof}

\vskip-.5cm

\begin{figure}[hbt]
\begin{center}
	\includegraphics[height=3.6cm]{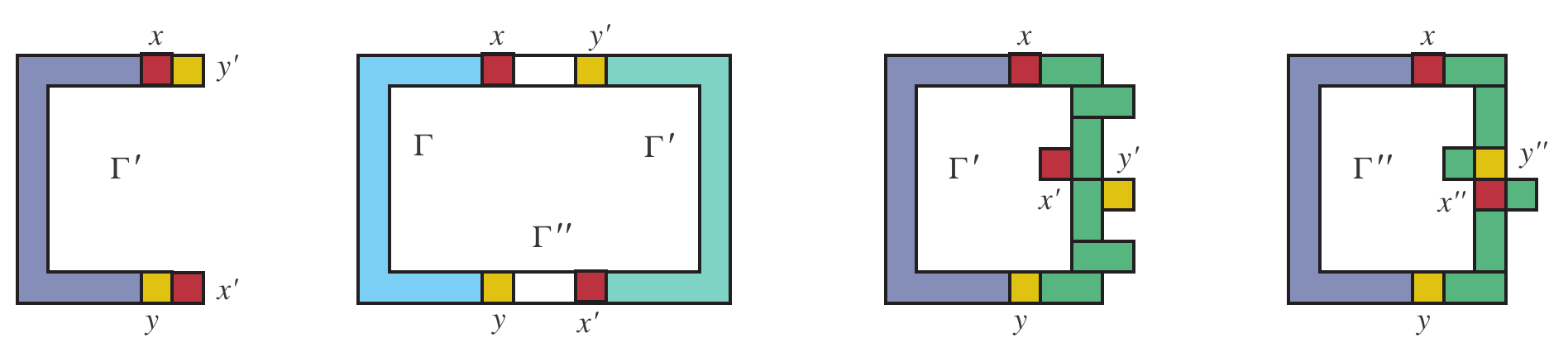}
\vskip-.3cm
\caption{Four transformations: \. $\Ga\in \cD(a,b) \ts \Rightarrow \ts \Ga'\in \cD(b,a)$, \
\ts $\Ga\in \cD(a,b)$, \ts $\Ga'\in \cD(a',b')  \ts \Rightarrow \ts \Ga''\in \cD(aa'+bb',bb')$, \
$\Ga\in \cD(a,b) \ts \Rightarrow \ts \Ga'\in \cD(0,a)$ \, and \, $\Ga\in \cD(a,b) \ts \Rightarrow \ts \Ga''\in \cD(a,0)$.}
\label{f:ext}
\end{center}
\vskip-.3cm
\end{figure}

\smallskip

\begin{cor} \label{c:domino-single}
Let \ts $\cD'(a,b)$ \ts be the set of regions \ts $\Ga$ such that
\ts $\tau(\Ga)=a$ \ts and \ts $\tau(\Ga-x-y) = b$ \ts for some domino \ts
$(x,y)$, where \ts $x,y\in \Ga$.  Then \ts $\cD'(a,b) \ne \emp$, \ts
for all \. $0\le b \le a$.
\end{cor}

\begin{proof}
Note that if \ts $(x,y)$ \ts is a domino in~$\Ga$, then \ts
$\tau(\Ga-x-y) \le \tau(\Ga)$. Thus the assumption \ts $b\leq a$ \ts
in the claim.  Now, for \ts $\Ga \in \cD(a-b,1)$ \ts where \ts $b \ge 1$, and \ts
$\Ga'\in \cD(1,b)$, the second transformation in \eqref{eq:two}
gives a region \ts $\Ga''\in \cD(a,b)$ \ts as in Figure~\ref{f:ext}.
Removing one white domino and keeping the other, gives the desired
region in \ts $\cD'(a,b)$.

Similarly, taking a region \ts
$\Ga \ssu \zz^2$ \ts with \ts $\tau(\Ga) = a$, attaching a domino
$(x,y)$ to a top right square \ts $z \in \Ga$ \ts gives region \ts
$\Ga'\in \cD'(a,0)$ \ts since \ts $\tau(\Ga+x+y)=\tau(\Ga)$ \ts
and \ts $\tau(\Ga-x-z)=0$.  The details are straightforward.
\end{proof}

\smallskip

\subsection{Slab tilings}\label{ss:domino-slab}
Denote by \ts $\cT_s(n)$ \ts the set of numbers of tilings with slabs:
$$
\cT_s(n) \ := \ \big\{\. \tau_s(\Ga)\., \ \ \text{where} \ \ \Ga \ssu \zz^3, \ |\Ga|\ts = \ts 4n\.\big\}.
$$


\begin{thm}  \label{t:slab-inverse}
There is a constant \ts $c>1$, such that
\. $\cT_s(n) \supseteq \{0,1,\ldots,c^n\}$,
for all \ts $n\ge 1$.  Moreover, for all \ts $k \le c^n$, a region \ts
$\Ga\ssu \zz^3$ \ts with \ts $\tau_s(\Ga)=k$ \ts and \ts $|\Ga|=2n$, can be constructed in time
polynomial in~$n$.
\end{thm}

\smallskip

Note that the corresponding result for the set \ts $\cT_b(n)$ \ts
of numbers of tiling with bricks, follows trivially from Theorem~\ref{t:domino}
since \. $\cT(n) \subseteq \cT_b(n)$.

\begin{proof}[Proof of Theorem~\ref{t:slab-inverse}]  The result follows
from the proof of Theorem~\ref{t:domino}.  Indeed, for every region \ts
$\Ga\ssu \zz^2$ \ts we can take a 2-layered region \ts
$\Ga_2:=\Ga \times \{0,1\} \ssu \zz^3$. Assuming \ts $\Ga$ \ts does not
have a \ts $2\times 2$ \ts square inside, we have \ts $\tau_s(\Ga_2) = \tau(\Ga)$.
The result now follows from reductions in Figure~\ref{f:imply}, where the \ts
$2\times 2$ \ts square in the middle is replaced by a \ts $3\times 3$ \ts
square without a center square (see an example below). Note that the notion
of $C$-triples  also needs to be  adjusted accordingly.  The details are straightforward.
\end{proof}
%
\vskip-.8cm
\begin{figure}[hbt]
\begin{center}
	\includegraphics[height=1.8cm]{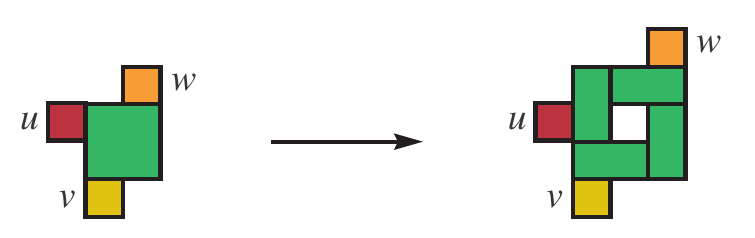}
%
\end{center}
\end{figure}
\vskip-.1cm

\medskip


\section{Complexity of coincidences}\label{s:perm}

\subsection{Parsimonious reductions} \label{ss:perm-slab}
Let \ts $f\in \SP$ \ts be a counting function.  As in the introduction, denote by
$$\text{\sc C}_f \ := \ \big\{\ts f(x) \. = ^? \ts f(y) \ts \big\}
$$
the \defn{coincidence problem} \ts for~$f$ \ts (cf.~\ref{ss:finrem-term}). This problem naturally
belongs to the complexity class
\begin{equation}\label{eq:CEP}
\CEP:=\big\{\ts f(x) \. =^? \ts g(y) \ \text{where} \ f, \. g \ts \in \ts \SP \ts \big\}.
\end{equation}
Note that \ts $\coNP \subseteq \CEP$, by definition.
The \ts {\sc \#3SAT} \ts coincidence problem \. {\sc C}$_{\text{\sc \#3SAT}}$ \. is a complete
problem in this class.

Function \ts $f$ \ts is said to have a \defn{parsimonious reduction} \ts to~$g$,
if there is an injection \ts $\rho: x\to y$ \ts from the instances of~$f$ to the instances of~$g$,
which maps values of the functions: \. $f(x)=g(y)$, and such that $\rho$ can be computed
in polynomial time. If \ts {\sc \#3SAT} \ts has a
parsimonious reduction to~$f$, then the coincidence problem \ts {\sc C$_f$} \ts is
\ts $\coNP$-hard because the decision problem \ts {\sc 3SAT} \ts is \ts $\NP$-complete.
Moreover, \ts {\sc C$_f$} \ts is \ts $\CEP$-complete by the definition of \ts $\CEP$.
In general, there are other ways for a problem to be
 \ts $\CEP$-complete, but having a parsimonious reduction is the most straightforward.

\smallskip

\begin{prop} [{\rm cf.\ Theorem~\ref{t:main-first}, part~{\small $(0)$}}{}]
\label{p:pars}
Let \ts $f\in \SP$ \ts be a function with a parsimonious reduction from \ts \text{\sc \#3SAT}.
Then the coincidence problem \ts {\sc C}$_f$ \ts is not in \ts $\PH$, unless \ts $\PH=\Sigmap_m$
\ts for some $m$.
\end{prop}

\begin{proof}
Suppose that \ts {\sc C}$_f \in \PH$.   Then \ts {\sc C}$_f \in \ts \Sigmap_n$ \ts
for some~$n$.  We then have:
\begin{equation}\label{eq:incl-collapse-CEP}
\PH \. \subseteq \. \NP^{\CEP} \. \subseteq \.
\NP^{\<\text{\sc C}_{\text{\sc \#3SAT}}\>} \. \subseteq \.
\NP^{\<\text{{\sc C}}_f\>} \. \subseteq \.
\NP^{\Sigmap_n} \. \subseteq \. \Sigmap_{n+1}\,,
\end{equation}
where the first inclusion is by Tarui \cite{Tar91} (see also \cite{Gre93}),
the second inclusion follows since \. {\sc C}$_{\text{\sc \#3SAT}}$ \. is a complete
problem in~$\ts \CEP$, the third follows from the parsimonious reduction of \ts
\text{\sc \#3SAT} \ts to~$f$,
and the last inclusion is by definition.  This proves the
desired collapse of the polynomial hierarchy.
\end{proof}

\smallskip
\begin{proof}[Proof of Theorem~\ref{t:slab-main}]
The $\NP$-completeness of \ts {\sc SlabTileability} \ts is proved in~\cite{PY}
by a bijective reduction from \ts {\sc 1-in-3 SAT}.  Since \emph{Schaefer's reduction} \ts
of \ts {\sc 1-in-3 SAT} \ts from \ts {\sc 3SAT} \ts is also by a bijection
\cite{Schaefer} (see also \cite[Problem~LO4]{GJ79}), we conclude that the
number of slab tilings has a parsimonious reduction from \ts {\sc \#3SAT}.
The result now follows from Proposition~\ref{p:pars}.
\end{proof}

\smallskip

\begin{proof}[Proof of Theorem~\ref{t:main-first}, parts~{\small $(1)$, $(2)$}]
These follow immediately from Proposition~\ref{p:pars} and
parsimonious reductions from \text{\sc \#3SAT} \ts given in~\cite{GJ79}.
\end{proof}
\smallskip

\subsection{Pattern containment} \label{ss:perm-pattern}
Let \ts $\pi \in S_k$ \ts and \ts $\si \in S_n$.  We say that \ts
$\si$ \ts  \defn{contains} \ts $\pi$ \ts if there is \ts
$A=\{a_1< \ldots < a_k\} \ssu [n]$, such that \ts $\si(a_1),\ldots,\si(a_k)$ \ts
has the same relative order as~$\ts \pi$.  Denote by \ts $\pc_\pi(\si)$ \ts
the number of subsets~$A$ as above. This counting functions is well studied
in various settings, see e.g.\ \cite{Kit,Vat}.

Define the \defn{pattern containment coincidence problem}:
$$\text{\sc C}_{\text{\sc PC}} \,= \,  \big \{\ts \pc_\pi(\si) \. =^? \ts \pc_{\pi'}(\si')\ts \big\}.
$$

\begin{proof}[Proof of Theorem~\ref{t:main-first}, part~{\small $(3)$}]
It was shown in~\cite[Cor.~4]{BBL}, that computing the pattern containment
function \ts $\pc_\pi(\si)$ \ts
is $\ts\SP$-complete, and the proof uses a parsimonious reduction from
\ts {\sc \#3SAT}.  Now the theorem follows from Proposition~\ref{p:pars}.
\end{proof}

\smallskip

\subsection{Kronecker coefficients} \label{ss:perm-Kron}
Consider the following application of the tools above to
algebraic combinatorics.
Let \ts $g(\la,\mu,\nu)$, where $\la,\mu,\nu\vdash n$,
denote the \defn{Kronecker coefficients}:
$$
g(\la,\mu,\nu) \, := \, \<\chi^\la\chi^\mu,\chi^\nu\> \, = \, \frac{1}{n!} \.
\sum_{\si \in S_n} \. \chi^\la(\si) \. \chi^\mu(\si)\. \chi^\nu(\si)\.,
$$
where \ts $\chi^\la$ \ts denote the irreducible $S_n$ character corresponding
to~$\la\vdash n$.  By definition, \ts $g(\la,\mu,\nu)\in \nn$ \ts and is symmetric
with respect to permutations of \ts $\la,\mu,\nu$. It is known that
\ts $g$ \ts is in \ts $\GapP := \SP-\SP$, i.e.\ can be written as the difference of two
functions in~$\SP$.  It is a major open problem whether \ts $g$ \ts is in~$\ts \SP$.
Here and below we are assuming that partitions are given in \emph{unary}.
We refer to recent surveys \cite{Pak-OPAC,Panova} for further background.

Define the \defn{Kronecker coefficients coincidence problem}:
$$\text{\sc C}_{\text{\sc KRON}} \,= \,  \big \{\ts g(\la,\mu,\nu) \. =^? \ts g(\al,\be,\ga)\ts \big\}.
$$

\begin{proof}[Proof of Theorem~\ref{t:main-first}, part~{\small $(4)$}]
It was shown in~\cite{IMW17} that the \emph{Kronecker vanishing problem} \ts
$\{g(\la,\mu,\nu) = ^? 0\}$ \ts is \ts $\coNP$-hard.  This proves the first part.
For the second part, recall that computing \ts $g$ \ts is \ts $\SP$-hard
(in the unary) \cite{IMW17}, and that the original proof  gives an explicit
(although rather involved) parsimonious reduction from \ts {\sc \#3SAT} \ts
to~$\ts g$.  Now the theorem follows from Proposition~\ref{p:pars}.
\end{proof}

\smallskip

\subsection{The permanent} \label{ss:perm-gen}
We start with a more general problem which will be used to
demonstrate our approach.  Recall that the \ts {\sc 0/1 Permanent} \ts
is the benchmark $\SP$-complete problem, which corresponds to counting
perfect matchings in simple bipartite graphs.

Let \ts {\sc \PerC} \ts denote the \ts \defn{0/1 Permanent
Coincidence problem}:
$$
\text{\sc \PerC} \ := \
\big\{ \per(M) \, =^? \, \per(M') \. \mid \. M, \ts M' \in \cM_n \big\},
$$
where \ts $\cM_n$ \ts is the set of $n \times n$ \ts matrices with
entries in $\{0,1\}$.

\smallskip

\begin{thm}[{\rm = Theorem~\ref{t:main-combined}, part~{\small $(0)$}}{}]
\label{t:perm}
\ts {\sc \PerC} $\notin\PH$ \ts unless \ts $\PH =\Sigmap_m$ \ts
for some~$\ts m$.
\end{thm}

\begin{proof}
Let \ts {\sc \PerV} \ts denote the \ts {\em 0/1 Permanent
Verification problem}:
$$
\text{\sc \PerV} \ := \
\big\{ \per(M) \, =^? \, k \big\},
$$
where \ts $M$ \ts is a 0/1 matrix and \ts $k\in \nn$ \ts is given in binary.  We have:
\begin{equation}\label{eq:incl}
\ComCla{PH} \. \subseteq \. \P^{\SP} \. \subseteq \. \NP^{\<\text{\sc \PerV}\>}
\. \subseteq \. \NP^{\<\text{\sc \PerC}\>}.
\end{equation}
The first inclusion is Toda's theorem \cite{Toda}.  The second inclusion
follows because \ts {\sc 0/1 PERMANENT} \ts is \ts $\SP$-complete \cite{V1,V3}.
Indeed, transform every query to a $\SP$ oracle to a 0/1 permanent instance, guess the
value~$k$ of that permanent, then call \ts {\sc \PerV} \ts to check that this
guess is correct.\footnote{Here we are using Ryan Williams's approach in \ts \href{https://cstheory.stackexchange.com/a/53024/16704}{cstheory.stackexchange.com/a/53024/}}

For the third inclusion in~\eqref{eq:incl}, use Theorem~\ref{t:domino}
to construct a region \ts $\Ga\ssu \zz^2$ \ts of size \ts $|\Ga|=O(\log k)^c$ \ts
and with exactly~$k$ domino tilings.  The dual bipartite graph $G$ then has
exactly~$k$ perfect matchings.\footnote{We can also use the Brualdi--Newman construction
here (see Remark~\ref{r:Nadeau}), instead of Theorem~\ref{t:domino}. }  This corresponds to an instance of the
 0/1 permanent equal to~$k$.  Now call \ts {\sc \PerC} \ts to simulate~{\sc \ts \PerV}.

Finally, suppose \ts {\sc \PerC} $\in \PH$.  Then \ts {\sc \PerC} $\in \ts \Sigmap_n$ \ts
for some~$n$.  By \eqref{eq:incl}, this implies:
\begin{equation}\label{eq:incl-collapse}
\PH \. \subseteq \. \NP^{\<\text{\sc \PerC}\>} \. \subseteq \.  \NP^{\Sigmap_n} \. \subseteq \. \Sigmap_{n+1}\,,
\end{equation}
as desired.
\end{proof}

\smallskip

\begin{proof}[Proof of Theorem~\ref{t:brick-main}]
Recall that the problem of counting \ts $2\times 1 \times 1$ \ts brick
tilings is $\SP$-complete \cite[Thm~1.3]{PY}, see also~\cite{V2}.
Now, note that the \ts {\sc C$_{\text{BT}}$} \ts restricted to the case
of regions in~$\zz^2$ still satisfies conclusions of Theorem~\ref{t:domino}.
From this point, the proof follows verbatim the proof of Theorem~\ref{t:perm}.
\end{proof}

\smallskip

\begin{rem}
The proof in~\cite{PY} is via a parsimonious reduction from the number
of perfect matchings in $3$-regular bipartite graphs, which in turn
is proved \ts $\SP$-complete via a non-parsimonious reduction from the \ts
{\sc 0/1 PERMANENT} \ts \cite[Thm~6.2]{DL92}.  This is why we cannot proceed
by analogy with the proof of Theorem~\ref{t:slab-main}.
\end{rem}

%

\medskip

\section{Variations on the theme}\label{s:var}

In this section we use the tools that we developed to solve other coincidence problems.

\subsection{Complete functions}\label{ss:var-complete}
Let \ts $\cX = \sqcup \ts \cX_n$, where \ts $\cX_n\subseteq \{0,1\}^{\ga(n)}$ \ts be a
set of \defn{combinatorial objects}, which means that the input size \ts $\ga(n)$ \ts
is polynomial in~$n$.  A counting function \ts $f\in \SP$ \ts can be viewed as a
function \ts $f: \cX\to \nn$.  For example, let \ts $\cX$ \ts be the set of simple
graphs \ts $G=(V,E)$ \ts where \ts $n=|V|$, and let \ts $f(G)$ \ts be the number of
Hamiltonian cycles in~$G$.

Denote
$$\cT_f(n) \, := \, \{f(x) \. \mid \. x\in \cX_n\} \quad \text{and} \quad
\cT_f\, := \,  \{f(x) \. \mid \. x\in \cX\} 
$$
the set of all values of the function~$f$.  We say that \ts $f$ \ts is \defn{complete} \ts
if \. $\cT_f = \nn$.  We say that \ts $f$ \ts is \defn{almost complete} \ts
if \. $\cT_f$ \. contain all but a finite number of integers \ts $k\in \nn$.

For example, by Theorem~\ref{t:domino}, the number of domino tilings in the
plane is a complete function.  Trivial examples of complete functions include
the number of $3$-cycles in a simple graphs, or the number \ts $\inv(\si)$ \ts
of inversions in a permutation~$\si$.  Similarly, the number of spanning trees
in a simple graph is in \ts $\{0,1,3,4,\ldots\}$, and thus almost complete.

\smallskip

\begin{ex}[{\rm Linear extensions}{}] \label{ex:indep-LE}
Let \ts $P=(X,\prec)$ \ts be a finite poset with \ts $n=|X|$ \ts elements.
A \defn{linear extension} of $P$ is a bijection \ts $\rho: X \to [n]$,
such that \ts $\rho(x) < \rho(y)$ \ts for all \ts $x \prec y$.
Denote by \ts $e(P)$ \ts the number of linear extensions of~$P$.
We use \ts {\sc \#LE} \ts to denote the problem of computing \ts $e(P)$.
It is known that \ts {\sc \#LE} \ts is \ts $\SP$-complete \cite{BW91}.

Clearly, \ts $e(P)\ge 1$ \ts for all~$P$.  Note that the set of numbers
of linear extensions \ts $\cT_e = \{1,2,3,\ldots\}$, since \ts $e(C_k + C_1)=k$,
where we take a parallel sum of an element and a chain of length~$k$.
In particular, the function~$e$ is almost complete.
\end{ex}

\smallskip

\begin{ex}[{\rm Symmetric Kronecker coefficients}{}] \label{ex:sym-Kron}
For a slightly nontrivial example, consider the
\defn{symmetric Kronecker coefficient} \.
$g_s(\la) :=g(\la,\la,\la)$,  see \cite{PP22}.
Note that function \ts $g_s: \{\la\} \to \nn$ \ts is complete:
$$
g_s(1^2) \ts = \ts 0\ts, \quad g_s(1)\ts =\ts 1 \quad \text{and} \quad
g_s(4k, 2k) \. = \. k+1 \ \ \text{for all} \ \ k\ge 1,
$$
where \ts $(4k,2k)\vdash 6k=n$, see e.g.\ \cite{Ste}.
\end{ex}

\smallskip

\begin{ex}[{\rm Littlewood--Richardson coefficients}{}]\label{ex:LR-complete}
The \ts \defn{Littlewood--Richardson {\rm $($LR$)$} coefficients} \ts $c^\la_{\mu\nu}$
\ts can be defined as structure constants for the ring of \defng{Schur functions}:
$$
s_\mu \ts \cdot \ts s_\nu \ = \ \sum_{\la} \, c^\la_{\mu\nu} \. s_\la \.,
$$
see e.g.\ \cite[Ch.~7]{Sta-EC}.  It remains open whether computing \.
$c^\la_{\mu\nu}$ \. is $\ts \SP$-complete (in unary), see an extensive discussion in
\cite{Pak-OPAC} and \cite{Panova}.  We note that this is a complete function:
$$
c^{3}_{1, 1^2} \ts = \ts 0, \quad
c^{2}_{1, 1} \ts = \ts 1 \quad \text{and} \quad
c^{k(321)}_{k(21),\ts k(21)} \. = \. k+1 \ \ \text{for all} \ \ k\ge 1,
$$
where the last equality follows form \. $c^{321}_{21,21}=2$ \. combined with  \cite[Rem.~5.2]{Rassart}.
\end{ex}

\smallskip

\begin{ex}[{\rm Contingency tables}{}]\label{ex:CT-complete}
Let \ts $\ba = (a_1,\ldots,a_r)\in \nn^r$ \ts and \ts $\bb = (b_1,\ldots,b_s)\in \nn^s$.
Denote by \ts $\CT(\ba,\bb)$ \ts the
number of \defn{contingency tables} \ts $M=(x_{ij})\in \nn^{rs}$, defined by
$$\sum_{j=1}^{s} \ts x_{ij} \. = \. a_i \ \ \text{for all \ $i$}\ts,
\quad \sum_{i=1}^{r} \ts x_{ij} \. = \. b_j \ \ \text{for all \ $j$}\ts, \quad \text{and}
\quad  x_{ij}\ts \ge\ts 0 \ \  \text{for all  \ $i,j$}\ts.
$$
Note that \ts $\CT(\ba,\bb)\ge 1$ \ts for all \ts $|\ba| = |\bb|$\ts.

Computing the number \ts $\CT(\ba,\bb)$ \ts of contingency tables (with the input in unary)
is conjectured to be \ts $\SP$-complete \cite[$\S$13.4.1]{Pak-OPAC}. On the other hand,
\ts $\CT(\cdot)$ \ts is clearly a complete function, since e.g.\
\ts $\CT(\ba,\ba) = k+1$, for \ts $m=n=2$ \ts and
\ts $\ba=(k,k)$.  Using standard reductions this also implies that the
\defng{Kostka number} \ts is also a complete function, see e.g.\ \cite{PV}.
\end{ex}

\smallskip

\begin{ex}[{\rm Pattern containment}{}]\label{ex:var-pattern-complete}
Let \ts $\si \in S_n$.
Clearly, the pattern containment function \ts $\pc_{21}(\si)=\inv(\si)$ \ts is
complete.  The following result is a generalization:

\begin{prop}\label{p:pattern-complete}
Fix \ts $\pi\in S_k$\ts, where \ts $k\ge 2$. Then function \ts $\pc_{\pi}(\si)$ \ts is
complete.
\end{prop}

\begin{proof}
Without loss of generality, assume that \ts $\pi$ \ts starts with an ascent.
Let \ts $\si$ \ts be a decreasing sequence \ts $m,m-1,\ldots,\pi(1)=a$. Append
it with \ts $\pi(2),\ldots,\pi(k)$ \ts which are shifted above \ts $m$ \ts if
they are strictly greater than $a$.  Suppose exactly \ts $r$ \ts elements are shifted.
Let \ts $m=n-r$ \ts be so that the resulting $\ts \si$ \ts is in  \ts $S_n$\ts,
and observe that \ts $\pc_\pi(\si) = n-r-a+1$.  This implies the result.

For example, let \ts $\pi=(2,5,1,3,6,4)$.  The construction above gives:
$$
\pc_{251364} \ts (n-4,n-5,\ldots,\ts 3,\ts 2,n-1,\ts 1,n-3,n,n-2) \. = \. n-5,
$$
where \ts $k=6$, \ts $a=2$, \ts $r=4$, and \ts $m=n-4$.
\end{proof}
\end{ex}

\smallskip

\subsection{Concise functions}\label{ss:var-concise}
Note that being complete is neither necessary nor sufficient for our
approach to the complexity of the coincidence problems. The following
purely combinatorial definition gets us closer to the goal.
\smallskip

\begin{Def} \label{d:concise}
Function \ts $f:\cX\to \nn$ \ts is called \defn{concise} \ts if there exist
some fixed constants \ts $C, c>0$, such that for all \ts
$k \in \cT_f$ \ts there is an element \ts $x\in \cX_n$ \ts with \ts
$f(x)=k$ \ts and \ts $n < C \ts (\log k)^c$.
\end{Def}

\smallskip

Our Theorem~\ref{t:domino} shows that number \ts $\tau(\Ga)$ \ts
of domino tilings of regions \ts $\Ga \ssu \zz^2$ \ts is concise.
On the other hand, the numbers of patterns \ts $\pc_\pi(\si) \le \binom{n}{k}$ \ts
for every \ts $\pi \in S_k$ \ts and \ts $\si\in S_n$
(see Example~\ref{ex:var-pattern-complete}), so the function
\ts $\pc_\pi$ \ts is not concise for a fixed~$\pi$.
We now present several less obvious examples of concise functions.


\smallskip

\begin{ex}[{\rm Independent sets}{}] \label{ex:indep-sets}
Let \ts $G=(V,E)$ \ts be a finite simple graph, and let \ts
$\la(G)$ \ts be the number of \defn{independent sets} in~$G$,
i.e.\ subsets \ts $X\subseteq V$ \ts such that $X$ contains
no two adjacent vertices.  Recall that computing \ts $\la$ \ts is \ts
$\SP$-complete, see \cite{PB83}.   Moreover, this holds even
for planar bipartite graphs \cite{Vad}.

We now show that \ts $\la$ \ts is concise.
Denote by $G'$ the graph obtained from~$G$
by adding a new vertex $w$ and adding all edges from~$w$ to $V$.
Similarly, denote by $G''$ the graph obtained from~$G$
by adding a new vertex $w$ disconnected from~$V$.
Observe that \.
$\la(G')=\la(G)+1$ \. and \. $\la(G'')=2\la(G)$.  Iterating these
two operations we obtain the desired graph $G$ on $n$~vertices,
with \ts $\la(G)=k$ \ts and \ts $n = O(\log k)$.
\end{ex}


\begin{ex}[{\rm Order ideals}{}] \label{ex:order-ideals}
Let \ts $P=(X,\prec)$ \ts be a finite poset with \ts $n=|X|$ \ts elements.
A subset \ts $Y \subseteq X$ \ts is a \defn{lower order ideal} \ts if
for all \ts $x\prec y$ \ts where \ts $x\in X$ \ts and \ts $y \in Y$,
we also have \ts $x\in Y$.  Denote by \ts $\mu(P)$ \ts the number of
lower order ideals in~$P$.  Recall that computing \ts $\mu$ \ts is \ts
$\SP$-complete, see \cite{PB83}.

We can similarly show that \ts $\mu$ \ts is concise by constructing
posets \ts $P', \ts P''$ \ts  with an extra element~$x$ that is either
smaller than all elements in~$X$, or incomparable to~$X$.  We then have
\. $\mu(P')=\mu(P)+1$, \. $\mu(P'') = 2\mu(P)$.  Iterating these two
operations proves the claim.

\begin{rem}\label{r:order-topologies}
For closely related problems on the smallest \defng{topology} \ts with a
given number of open sets, and the shortest \defng{addition chain} \ts
of a given integer, see \cite{RT10}, \cite[$\S$4.6.3]{Knu98}, and sequences
\cite[\href{http://oeis.org/A137814}{A137814}]{OEIS},
\cite[\href{http://oeis.org/A003064}{A003064}]{OEIS}.
\end{rem}
\end{ex}

\smallskip

\begin{ex}[{\rm Satisfiability}{}] \label{ex:SAT}
Recall that satisfiability decision problem \ts {\sc 2SAT} \ts is in~$\P$,
while the corresponding counting problems \ts {\sc \#MONOTONE 2SAT} \ts
is $\SP$-complete \cite{V1}.  Here \emph{monotone} \ts refers to boolean
formulas which do not have negative variables.

Observe that the number of independent sets \ts $\la(G)$ \ts
has an obvious parsimonious reduction to \ts {\sc \#MONOTONE 2SAT}:
$$G=(V,E) \ \ \longrightarrow \ \ \Phi_G \, := \, \bigwedge_{(v,w)\in E} \. (x_v \vee x_w),
$$
since complements to independent sets \ts $V\sm X$ \ts are in natural
bijection with satisfying assignments of~$\Phi_G$. Since \ts $\la(G)$ \ts
is concise, then so is \ts {\sc \#MONOTONE 2SAT}.  Similarly, \ts {\sc \#3SAT} \ts is also concise,
via the standard reduction:
$$
G=(V,E) \ \ \longrightarrow \ \ \ \Phi_G' \, := \, (z\vee z \vee z) \bigwedge_{(v,w)\in E} (x_v \vee x_w \vee \ov z).
$$
\end{ex}

\smallskip

The approach in this example can be distilled in the following basic observation:

\smallskip

\begin{prop}\label{ss:SAT-reduction}
Suppose a concise counting function \ts $g$ \ts has a parsimonious reduction to~$\ts f$.  Then \ts $f$ \ts is also concise.
\end{prop}

\smallskip

\subsection{Further examples} \label{ss:var-more}
We start with the following general notion of exponential growth
which will prove useful in several examples.\footnote{This definition
is somewhat non-standard as we make no distinction between weakly exponential,
exponential, factorial and superexponential growths:
\ts $e^{\sqrt{n}}$, \ts $e^{n}$, \ts $e^{n\log n}$ \ts and \ts $e^{n^2}$.
We refer, e.g., to \cite{Grig} for a more refined treatment of growth functions.}

\smallskip

\begin{Def}\label{d:exp-growth}
We say that a counting function \ts $f:\cX\to \nn$ \ts has \defn{exponential growth} \ts
if there exist \ts  $A,\al>0$, such that \ts $f(x) > A \exp(n^\al)$ \ts for all \ts $x\in \cX_n$\ts.
\end{Def}

\smallskip

For example, the number of independent sets of a bipartite graph~$\ts G \ts $ on $\ts n \ts $ vertices
has exponential growth:  \ts $\la(G) \ge 2^{n/2}$. Note also that every \ts $f \in \SP$ \ts
satisfies the opposite inequality: \ts $f(x) < B \exp(n^\be)$ \ts
for all \ts $x\in \cX_n$ \ts and some \ts  $B,\be>0$.

\smallskip

\begin{prop} \label{p:concise-complete}
Let \ts $\cX = \sqcup \ts \cX_n$ \ts be a set of combinatorial objects.
Suppose function \ts $f:\cX\to \nn$ \ts has exponential growth.
Suppose also that \ts $f$ \ts is almost complete.  Then \ts $f$ \ts is concise.
\end{prop}

\begin{proof}
Since \ts $f$ \ts is almost complete, there exists an integer \ts $N$, s.t.\
for every \ts $k>N$, we have \ts $f(x) = k$ \ts for some \ts $x\in \cX_n$.
Since \ts $f$ \ts has exponential growth, we must have \ts $k > A \exp(n^\al)$.  Let
$$C \, := \, \max \big\{\ts n \. : \. f(x) \le N \ \ \text{for some} \ \ x\in \cX_n \ts \big\}.
$$
Then, \ts
$n < C + \big(\log \frac{k}{A}\big)^{1/\al}$ \ts for all \ts $k$, so
\ts $f$ \ts is concise.
\end{proof}

\smallskip

\begin{ex}[{\rm Linear extensions of restricted posets}{}] \label{ex:LE-restricted}
%
The \defn{height} \ts and \defn{width} \ts of a \ts $P$ \ts is the size of the
longest chain and antichain, respectively.  For posets of bounded width,
\ts {\sc \#LE} is in~$\P$ (via \emph{dynamic programming}).
On the other hand, for posets of height two,  \ts {\sc \#LE} \ts is
\ts $\SP$-complete \cite{DP18}.

For a permutation \ts
$\si \in S_n$ \ts consider a partial order \ts $P_\si = ([n], \prec)$, where \ts
$i \prec j$ \. if and only if \. $1\le i<j\le n$ \ts and \ts $\si(i)<\si(j)$.
Such posets \ts $P_\si$ \ts are called \defn{two-dimensional}, and
\ts {\sc \#LE} \ts is also \ts $\SP$-complete for this family \cite{DP18}.
Note that all posets of width two also  have dimension two, see e.g.~\cite{Tro}.

In \cite{KS21}, Kravitz and Sah show that function~$e$ is concise on
posets of width two (see also~\cite{CP-CF}).  Additionally, they prove that
\begin{equation}\label{eq:KS}
\cT_e(n)  \, \supseteq \, \big \{1,\ldots, c^{n/(\log n)} \big\} \quad
\text{for some} \ \ c>1.
\end{equation}
In particular, they prove a \ts $O(\log k \log \log k)$ \ts bound on the minimal
size of a poset with $k$ linear extensions, cf.\ \ts
\cite[\href{http://oeis.org/A160371}{A160371}]{OEIS} \ts and \ts
\cite[\href{http://oeis.org/A281723}{A281723}]{OEIS}.  They also conjecture
the \ts $O(\log k)$ \ts bound for posets of width two, and thus for
general posets \cite[Conj.~$7.3$ and~$7.4$]{KS21}.

On the other hand, for posets of height two, we have \.
$\lfloor n/2\rfloor! \cdot \lceil n/2\rceil! \ts \le \ts e(P) \ts \le \ts n!$\ts.
Thus, we have \. $k = \exp \big(n \log n + O(n)\big)$, i.e.\
function \ts $e$ \ts restricted to height two posets has exponential growth.
Now Proposition~\ref{p:concise-complete} gives a somewhat unexpected result:

\smallskip

\begin{prop} \label{p:LE-heigh-two}
If the function \ts $e$ \ts restricted to height two posets is almost complete,
then it is also concise.
\end{prop}

This suggests the following conjecture:


\begin{conj} \label{conj:LE-height-two}
The function \ts $e$ \ts restricted to height two posets is almost complete.
\end{conj}

\smallskip

By Proposition~\ref{p:LE-heigh-two} this gives a strong improvement over \eqref{eq:KS}
for general posets:

\smallskip

\begin{prop}\label{p:LE-height-two-asy}
Conjecture~\ref{conj:LE-height-two} implies that
\begin{equation}\label{eq:KS-conj}
\cT_e(n)  \, \supseteq \, \big \{1,\ldots, e^{n\log n - c\ts n} \big\} \quad
\text{for some} \ \ c>0.
\end{equation}
\end{prop}

In particular, Conjecture~\ref{conj:LE-height-two} would imply
\cite[Conj.~$7.3$]{KS21} mentioned above.   See~$\S$\ref{ss:finrem-LE-height-two}
for more on this conjecture, and \cite{CP-CF} for connections to other 
conjectures in number theory.\end{ex}


\smallskip

\begin{ex}[{\rm Matchings}{}] \label{ex:var-match}
Let \ts $G=(V,E)$ \ts be a simple graph on \ts $n=|V|$ \ts
vertices.  A \defn{matching} \ts is a subset \ts $X\subseteq E$ \ts
of pairwise nonadjacent edges.  Denote by \ts $\ma(G)$ \ts number of
matchings. In~\cite{Jer87}, Jerrum showed that computing \ts $\ma$ \ts is
$\SP$-complete, even when restricted to planar graphs.  Vadhan
extended this to planar bipartite graphs \cite{Vad}.
The following result shows that \ts $\ma$ \ts is concise
even when restricted to forests.

\begin{prop}  \label{p:match-forest}
The function \. \. $\ma$ \. restricted to forests is concise.
\end{prop}

\begin{proof}  By analogy with the proof of Theorem~\ref{t:domino}, let
\ts $\cD(a,b)$ \ts be the set of trees $T$ such that \ts $\ma(T)=a$ \ts
and \ts $\ma(T-x)=b$ \ts for some vertex~$x$ in~$T$.  Let $T'$ be a tree
obtained by adding a vertex~$y$ and an edge~$(xy)$.  Clearly, \ts $\ma(T')=a+b$ \ts
and \ts $\ma(T'-x) = b$.  We conclude:
$$
\cD(a,b) \ne \emp \ \ \Longrightarrow  \ \ \cD(a+b,a) \ne \emp \quad \text{and} \quad
\cD(a+b,b) \ne \emp,
$$
where in the first implication we have \ts $(T,x) \to (T',y)$,
and the second implication we have \ts $(T,x) \to (T',x)$.
Iterating this procedure starting with a single edge, we obtain a tree $T_{a\ts b}$
for every relatively prime $(a,b)$, $a\ge b \ge 1$.  Following~\cite{KS21}, we can
think of pairs \ts $(a,b)$ \ts as vertices of the \defng{Calkin--Wilf tree},
and conclude that for every prime $a$ there is an integer~$b$, such that some tree
in \ts $\cD(a, b)$ \ts has \ts $O(\log a \ts \log \log a)$ \ts vertices.

For a given integer~$k$, take a prime factorization \ts $k=a_1 \cdots a_\ell$ \ts
and let \ts $F$ \ts be the union of the corresponding trees \ts $T_i \in \cD(a_i,b_i)$.
Continuing the analysis in~\cite{KS21}, we conclude that the forest \ts $F$ \ts has
\ts $O(\log k \ts \log \log k)$ \ts vertices and \ts $\ma(F)=k$.
\end{proof}
\end{ex}

\smallskip

\begin{ex}[{\rm Spanning trees}{}] \label{ex:var-st}
Let \ts $G=(V,E)$ \ts be a simple graph and let \ts $\st(G)$ \ts
be the number of spanning trees in~$G$.  Only recently, Stong
proved in~\cite{Stong} that function \ts $\st$ \ts is concise using a
technical number theoretic argument.  This resolved an open
problem which goes back to~\cite{Sed}.  More precisely, Stong proved that
for all \ts $k$ \ts sufficiently large, there is a simple \emph{planar} \ts  graph
\ts $G$ \ts on $n$ vertices with exactly \ts $k$ \ts spanning trees,
and \ts $n < C \ts (\log k)^{3/2}/(\log \log k)$. Note also that
the function \ts $\st \in \FP$ \ts since it can be computed by
the \defng{matrix-tree theorem}.
\end{ex}

\smallskip

\begin{ex}[{\rm Pattern containment}{}] \label{ex:var-PC}
As we discussed in~$\S$\ref{ss:perm-pattern}, the pattern
containment function \ts $\pc(\si,\pi)$ \ts has a parsimonious
reduction from \ts {\sc \#3SAT}.  From above, this immediately
implies that function~$\ts \pc \ts$ is concise.  It would be
interesting to see a more direct construction of this result.
\end{ex}
\smallskip

\begin{ex}[{\rm Young tableaux}{}] \label{ex:sym-YT}
Denote by \ts $\SYT(\la/\mu)$ \ts the number of \defng{standard Young tableaux} \ts
of a skew shape~$\la/\mu$.  Recall that computing \ts $\SYT(\la/\mu)$ \ts is in~$\FP$
by the \emph{Aitken--Feit determinant formula}, see e.g.\ \cite[Eq.~$(7.71)$]{Sta-EC}.
Note also \ts $\SYT$ \ts is almost complete, since \ts $\SYT(n-1,1) = n-1$.

\begin{question}\label{q:SYT-skew}
Is \. $\SYT$ \ts concise?  In other words, does for all \ts $k>0$ \ts
there exist partitions \ts $\mu\ssu \la $ \ts which satisfy
\ts $\SYT(\la/\mu)=k$ \ts and \ts  \ts $|\la/\mu| \le C (\log k)^c$,
for some fixed \ts $C,c>0$?
\end{question}

By definition, the function \ts $\SYT$ \ts is a restriction of the function~$e$
counting the number of linear extensions to two-dimensional posets corresponding
to skew shapes.  Therefore, if the answer to the question is positive, the proof will
be very challenging (cf.~$\S$\ref{ss:finrem-reduced}).  There is, however, some negative evidence.

\begin{prop}\label{p:SYT-concise}
Function \ts $\SYT$ \ts restricted to straight shapes $($i.e.~$\mu=\emp)$, is \ts \underline{not} \ts  concise.
\end{prop}

\begin{proof}  Recall that \. $\SYT(\la) = \chi^\la(1) \ts| \ts n!$ \. for all \ts $\la \vdash n$.
Thus \ts $\SYT(\la)$ \ts cannot be a prime $>n$.
Since \ts $\SYT$ \ts is almost complete even when restricted to straight shapes, it is not concise.
\end{proof}

Note that the argument in the proof does not apply to general skew shapes,
as we can have large prime even for the \emph{zigzag shape} \ts $\rho_k/\rho_{k-2}$,
where \ts $\rho_k=(k,k-1,\ldots,1)$, see~$\S$\ref{ss:finrem-LE-height-two}.
Note also that function \ts $\SYT$ \ts restricted to straight shapes can have unbounded
coincidences (beyond conjugation).  This was proved in \cite{Cra} by an elegant
construction.


\begin{prop}\label{p:sym-SYT}
Function \ts $\SYT$ \ts  restricted to self-conjugate straight shapes $($i.e.~$\mu=\emp$
and  $\la=\la')$, is \ts \underline{not} \ts almost complete.
\end{prop}

\begin{proof}
Observe that \ts $\SYT$ \ts restricted to self-conjugate straight shapes,
has exponential growth: \. $\SYT(\la) \ge 2^{n(1-o(1))}$ \.
for all \ts $\la=\la'\vdash n$.  Indeed, let  \ts $h_{ii}(\la)=2a_i+1$,
\ts $1\le i \le d$,  be the \emph{principal hook length} \ts of~$\la$,
where \ts $d$ \ts is the size of the \emph{Durfee square} \ts of~$\la$.
We have:
$$
\SYT(\la) \, \ge \, \binom{2a_1}{a_1}  \cdots \binom{2a_d}{a_d} \, \ge \,
\frac{2^{2a_1-1}}{\sqrt{a_1}} \, \cdots \, \frac{2^{2a_d-1}}{\sqrt{a_d}}
 \, \ge \, \frac{2^{n-2d}}{n^{d/2}}  \, \ge \, \frac{2^{n-2\sqrt{n}}}{n^{\sqrt{n}}}
 \, \ge \, 2^{n\ts - \ts O(\sqrt{n} \ts \log n)} \,,
$$
since \ts $d\le\sqrt{n}$ \ts and \. $\binom{2k}{k} \ge 2^{2k-1}/\sqrt{k}$ \.
for all \ts $k\ge 1$.
On the other hand, there are at most \. $n \ts p(n)=e^{O(\sqrt{n})}$ \. possible
values of \ts $\SYT$ \ts on symmetric partitions of size at most~$n$.
The disparity with the growth implies the result.
\end{proof}
\end{ex}

\smallskip

\begin{ex}[{\rm Kronecker coefficients}{}] \label{ex:var-Kron}
As we discussed in~$\S$\ref{ss:perm-Kron}, the Kronecker
coefficients function \ts $g(\la,\mu,\nu)$ \ts has a parsimonious
reduction from \ts {\sc \#3SAT}.  From above, this immediately
implies that function~$\ts g \ts$ is concise.

On the other hand, the symmetric Kronecker coefficients function \.
$g_s(\la)=g(\la,\la,\la)$ \.  is more mysterious.  Although \ts $g_s$ \ts
is complete (see Remark~\ref{ex:sym-Kron}), it was proved only recently
\cite[Thm~1.3]{PP20}, that \.
$\max_{\la\vdash n} \ts g_s(\la) \ts = \ts \exp\Omega(n^\al)$, via
an explicit construction based on the approach in~\cite{IMW17}.
The exact asymptotics \. $\max_{\la\vdash n} \ts g_s(\la) \ts = \ts \exp\Theta(n \log n)$ \.
was given in~\cite{PP22} by a nonconstructive argument.

\begin{conj} \label{conj:Kron-concise}
Symmetric Kronecker coefficient function~$g_s$ is concise.
\end{conj}
\end{ex}

\smallskip

\begin{ex}[{\rm Contingency tables}{}]\label{ex:CT-concise}
In notation of Example~\ref{ex:CT-complete}, denote by \ts $n:= |\ba| = |\bb|$ \ts
the \defn{size} \ts of the contingency table, where \ts $\ba\in \nn^r$ \ts and \ts $\bb\in \nn^s$.
Recall that \ts $\CT$ \ts is complete.


\begin{conj}  \label{conj:CT}
The function \ts $\CTr$ \ts is concise, i.e.\ for every \ts $k>0$ \ts there
exist vectors \ts $\rba,\rbb$ \ts of size $n$, such that \ts $\CTr(\rba,\rbb) =k$ \ts
and \ts $n\le C \ts (\log k)^c$, for some fixed \ts $C,c>0$.
\end{conj}

Recall that
$$\CT(\ba,\bb) \, \le \, \left(1 + \frac{rs}{n}\right)^n \left(1 + \frac{n}{rs}\right)^{rs}
\, \le \, 4^n\.,
$$
where the second inequality is under assumption \ts $rs \le n$, see \cite[Thm~1.1]{PP20}.
Moreover, this upper bound is tight up to lower order terms.  On the other hand, the number
of pairs \ts $(\ba,\bb)\in \nn^{r+s}$ \ts is \ts $n^{O(r+s)}$, suggesting
that \ts $c\ge 2$ \ts in the conjecture.  We refer to \cite{Bar,BLP},
for an overview of the known bounds on \ts $\CT(\ba,\bb)$, and further references.
See also $\S$\ref{ss:finrem-deg-seq} and $\S$\ref{ss:finrem-LR}, for two closely related
problems.
\end{ex}

\smallskip

\subsection{Back to coincidence problems} \label{ss:var-pars}
Let us first summarize what we know.
Recall the notion of the coincidence problem \ts {\sc C}$_f$ \ts defined
in~\eqref{eq:CEP}.  When \ts $f \in\FP$, we trivially have \ts {\sc C}$_f \in \P$.
The examples include the number of standard Young tableaux of skew shapes,
spanning trees in graphs, pattern containment of a fixed pattern,
and linear extensions of width two posets.

When \ts $f$ \ts has a parsimonious reduction from \ts {\sc \#3SAT},
the complexity of \ts {\sc C}$_f$ \ts is given by Proposition~\ref{p:pars}.
The examples are given in Theorem~\ref{t:main-first}.
When \ts $f$ \ts is \ts \emph{not} \ts known to be either in \ts $\FP$ \ts nor \ts $\SP$-hard,
none of our tools apply.  The examples include the number of contingency tables,
the Littlewood--Richardson coefficients, and symmetric Kronecker coefficients
(additional examples are given in Section~\ref{s:finrem}).

Finally, as the examples above show, there are many cases when \ts $f$ \ts
is $\SP$-complete, but the corresponding decision problem is in~$\P$.
The examples  are given in the Main Theorem~\ref{t:main-combined}
which we are now ready to prove.

\smallskip

\begin{proof}[Proof of Main Theorem~\ref{t:main-combined}]
For \ts {\small $(1)$}, \ts {\small $(2)$} \ts and {\small $(3)$},
recall that these counting functions are concise, and that computing
them is $\SP$-complete (see above).  The rest of the proof follows
verbatim the proof of Theorem~\ref{t:perm}.

\smallskip

For \ts {\small $(4)$}, the problem \ts {\ts \#LE} \ts is $\SP$-complete,
the counting function \ts $e$ \ts concise, but the proof is a little less
straightforward.
Indeed, the proof in \cite{KS21} does not produce a polynomial time construction
of the width two poset \ts $Q$ \ts with \ts $n=O(\log k \ts \log \log k)$ \ts
elements and \ts $e(Q) = k$.  Even the first step requires factoring of~$k$
which not known to be in polynomial time (cf.\ Remark~\ref{r:KS-BPP}).

The way to get around the issue is to \emph{guess} \ts the width two poset~$Q$
of size \ts $n$ \ts with \ts $e(Q)=k$.  Such poset exists by \cite{KS21}, and
\ts $\{e(Q)=^? k\}$ \ts can be verified in polynomial time \ts
$O\big((\log k)^c\big)$ \ts since \ts $Q$ \ts
has width two and size  \ts $n=O(\log k \ts \log \log k)$.  Denote
$$\text{\sc V}_e \. := \. \{e(Q) =^? k\} \quad \text{and} \quad
\text{\sc C}_e \. := \. \{e(P) =^? e(Q)\},
$$
the verification and coincidence problems for the numbers of linear extensions.
Note that a version of \eqref{eq:incl} still holds in this case:
\begin{equation}\label{eq:incl-LE}
\ComCla{PH} \. \subseteq \. \P^{\SP} \. \subseteq \. \NP^{\<\text{\sc V}_e\>}
 \. \subseteq \. \NP^{\<\{e(P)=^?e(Q)\}\>, \. \<\{ e(Q)=^?k \}\>}\. \subseteq \. \NP^{\<\text{\sc C}_e\>},
\end{equation}
where the first inclusion is Toda's theorem again, in the third inclusion
we guess both~$Q$ and $k$, and in the fourth inclusion we use \ts
$\{ e(Q)=^?k \} \in \P$.  From this point, proceed verbatim the
proof of Theorem~\ref{t:perm}.

\smallskip

For  \ts {\small $(5)$}, we use the proof of Proposition~\ref{p:match-forest}
to guess a forest \ts $F$ \ts with \ts $n=O(\log k \log \log k)$ \ts vertices and
\ts $\ma(F) = k$ \ts matchings.   Since we are using the approach in \cite{KS21}
again, we similarly conclude that this cannot be used to construct an explicit
forest \ts $F$ \ts in polynomial time.  On the other hand, note from the proof
of the proposition, that give the forest \ts $F$ \ts and the order of removed
vertices, the function \ts $\ma(F)$ \ts can be computed by induction
in polynomial time.

We now follow the approach in \ts {\small $(4)$}.  Proceed by guessing the
desired forest \ts $F$ \ts and the order of vertices to be removed,
verify in polynomial time that this forest has exactly \ts $k$ \ts matchings,
and obtain the analogue of~\eqref{eq:incl-LE}.  Recall also that the problem
of computing the number of matchings is $\SP$-complete for bipartite planar
graphs (see above).  From this point, proceed as in \ts {\small $(4)$} \ts above.

\smallskip

For \ts {\small $(6)$}, recall that computing \ts $b(M)$ \ts is \ts $\SP$-complete
\cite{Snook}.  Recall also that graphic matroids are rational,
so \ts the number of spanning trees  \ts $\st$ \ts
is a restriction of the counting function~$b$.
We  need only a weaker version \cite[Cor.~6.2.1]{Stong}, which states
that for all \ts $k$ \ts sufficiently large, there is a simple graph
\ts $G$ \ts on $n$ vertices with exactly \ts $k$ \ts spanning trees,
and \ts $n < C \ts (\log k)^{2}$.  Examining the proof of this result
gives a polynomial time construction, but even without a careful
examination the verification problem \ts $\{\st(G)=^? k\}$ \ts
is clearly in \ts $\P$ \ts by the matrix-tree theorem.  Thus we obtain
the analogue of~\eqref{eq:incl-LE}, and can proceed as above.
\end{proof}

\begin{rem}  \label{r:KS-BPP}
The authors of \cite{KS21} reported to us\footnote{Personal communication (July, 2023).} that there is a way
to avoid integer factoring by separating primes into \emph{small}: \.
$p\le C (\log k)^{3/2}$, and \emph{large}: \. $p>C (\log k)^{3/2}$.
The former can be found exhaustively, while the
latter can treated probabilistically in one swoop.  This gives a
probabilistic polynomial time algorithm for generating a poset
with few elements and desired number of linear extensions.
We should mention that the best deterministic version of this
problem is known to give only $n=O(\sqrt{k})$, see \cite{Ten}.
\end{rem}

\smallskip

\subsection{Most general version} \label{ss:var-gen}
Now that we treated many examples, we are ready to state the most
general complexity result which can be used as a black box.
For that we need a new definition which combines the notions
of concise functions and the polynomial time complexity.

Let \. $\cX = \sqcup \ts \cX_n$ \. be the set of combinatorial objects
(see~$\S$\ref{ss:var-complete}), and let \. $\cY = \sqcup \ts \cY_n$ \.
be a subset, such that \. $\cY_n \subseteq \cX_n$\ts.  For a counting function \ts
$f:\cX\to \nn$ \ts denote \. $\cT_{f,\cY} = \{f(y) \mid y \in \cY\}$.
Similarly, denote by
\begin{equation}\label{eq:rec}
\text{\sc V}_{f,\cY} \, := \, \big\{\ts f(y) \ts =^? \ts k \. \mid \. y \in \cY_n \ \ \text{and} \ \ k \in f(\cX_n)
\ts \big\}
\end{equation}
the \defn{restricted verification problem}.

\smallskip

\begin{Def}\label{def:gen}
A counting function \ts $f:\cX\to \nn$ \ts
is called \defn{recognizable} \ts if there exists \.
\. $\cY = \sqcup \ts \cY_n\ts,$ \. $\cY_n \subseteq \cX_n$\ts,
such that \. $\cT_{f,\cX}=\cT_{f,\cY}$ \.
and \. $\text{\sc V}_{f,\cY} \ts \in \ts \PH$.
\end{Def}

\smallskip

Note that if \ts $f$ \ts is almost complete on~$\cX$,  then it is
almost complete on~$\cY$.  In this case, we can replace the \ts
``$k \in f(\cX_n)$'' \ts assumption with ``for all sufficiently large $k$''.
In order for the problem to be in $\PH$, the input of $y$ has to have
a polynomial in the bit-length of~$k$.  Thus, almost complete
recognizable functions must also be concise.

\smallskip

\begin{prop}\label{p:rec}
Let \ts $f$ \ts be a counting function which is both $\SP$-complete
and recognizable.  Then \. {\sc C}$_f$ \ts $\in \PH \ \Rightarrow \ \PH=\Sigmap_m$ \. for some $m$.
\end{prop}

\begin{proof} Since \ts $f$ is recognizable, we have \ts {\sc V}$_{f,\cY}\in \PH$ \ts
for some \ts $\cY\subseteq \cX$.  Then \ts {\sc V}$_{f,\cY}\in \Sigmap_\ell$ \ts
for some \ts $\ell\ge 1$. We have:
\begin{equation*}
\ComCla{PH} \. \subseteq \. \P^{\SP} \. \subseteq \. \NP^{\<\text{\sc V}_{f}\>}
\. \subseteq \. \NP^{\<\{f(x) =^? f(y)\}\>, \. \<\{f(y)=^?k\}\>}
\. \subseteq \. \NP^{\<\text{\sc C}_f\>, \. \<\text{\sc V}_{f,\cY}\>}
\. \subseteq \. \NP^{\<\text{\sc C}_f\>, \. \Sigmap_\ell}
\. \subseteq \. \Sigma_{\ell+1}^{\<\text{\sc C}_f\>},
\end{equation*}
where in the first inclusion we use Toda's theorem again,
in the third inclusion we guess both~$y$ and~$k$, and in the fourth inclusion we use
the definition of recognizable functions.

Now suppose \ts {\sc C}$_f \in \PH$.   Then \ts\ts {\sc C}$_f  \in \ts \Sigmap_r$ \ts
for some~$r$. Then we have:
\begin{equation}\label{eq:incl-gen-collapse}
\PH \. \subseteq \. \Sigma_{\ell+1}^{\<\text{\sc C}_f\>} \. \subseteq \.
\Sigma_{\ell+1}^{\Sigmap_r} \. \subseteq \. \Sigmap_{r+\ell+1}\,,
\end{equation}
as desired.
\end{proof}

\smallskip

\begin{rem}
For example, suppose $f$ is concise and the instance \ts $y\in \cY_n$ \ts
with given \ts $f(y)=k$ \ts can be obtained by a probabilistic polynomial
time algorithm as in Remark~\ref{r:KS-BPP}.  Since \. $\BPP\subseteq \Sigmap_2$,
Proposition~\ref{p:rec} applies in this case.

Similarly, suppose Conjecture~\ref{conj:LE-height-two} has a nonconstructive
proof, where the family of posets $\cY$ with some parameters are introduced,
and the almost completeness is proved by a probabilistic method.
Recall that \ts {\sc \#LE} \ts restricted to height two posets is $\SP$-complete.
By Proposition~\ref{p:LE-heigh-two}, the restriction of function \ts $e$ \ts
to height two posets is concise.  Therefore, if the number of random bits used
is at most polynomial, we can again apply Proposition~\ref{p:rec}.
\end{rem}

\medskip

\section{Final remarks and further open problems}\label{s:finrem}

\smallskip

\subsection{On terminology}\label{ss:finrem-term}
There is a gulf of a difference between the notions of ``equality'',
``equality conditions'' and ``coincidence''.  Deciding \defna{equality} \ts
is typically of the form \ts $\{f(x)=^?g(x) \ \forall x\}$, whether two functions
are equal \emph{for all~$x$}.  Deciding the \defna{equality conditions} \ts
is typically of the form \ts $\{f(x)=^?g(x)\}$, whether two functions
are equal \emph{for a given~$x$}.  Usually this refer to the inequality
\ts $f(x) \geqslant g(x)$ \ts which holds for all~$x$
(cf.~$\S$\ref{ss:finrem-equality-conditions}).
Finally, deciding the \defna{coincidence} \ts is typically of the form \ts $\{f(x)=^?f(y)\}$,
whether the same function has equal values at \emph{given} \ts $x,y$.
Despite superficial similarities, these properties
have very distinctive flavors from the computational complexity point
of view.  We refer to \cite{Wig19} for the general background and
philosophy.


\subsection{Two types of $\ts \SP$-complete problems}\label{ss:finrem-two-types}
There are so many known $\SP$-complete problems, it can be difficult to trace
down the literature to see if they are proved by a parsimonious reduction from \ts {\sc \#3SAT}
(often, they are not).  Helpfully, \cite{GJ79} addresses this for every reduction,
but \cite[$\S$13]{Pak-OPAC} does not, for example.  Clearly, a  parsimonious reduction
from \ts {\sc \#3SAT} \ts is not possible if the corresponding decision problem is in~$\P$.
This applies to the numbers of linear extensions, order ideals, independent sets,
matchings and bases of matroids considered in Section~\ref{s:var}, since
the corresponding decision problems are trivially in~$\ts \P$.

\smallskip

\subsection{Bases in matroids}\label{ss:finrem-matroids}
Recall that every graphic matroid is binary, so \defng{binary matroids} \ts is a more
natural class to be used in part~{\small $(6)$} of Main Theorem~\ref{t:main-combined}.
For the number of bases of binary matroids, the $\SP$-completeness result
was claimed by Vertigan and is widely quoted (see e.g.\ \cite{Wel93}),
despite not appearing in print (cf.\ \cite{Snook} and \cite[$\S$13.5.8]{Pak-OPAC}).
The complete proof was obtained most recently by Knapp and Noble \cite[Thm~50]{KN23},
after this paper was already written.

Let us mention that number of bases of matroids is \ts $\SP$-complete
for other concise presentations.  Notable classes include \defng{sparse paving matroids} \cite{Jer06}
(proved via a parsimonious reduction to the number of Hamiltonian cycles),
and \defng{bicircular matroids} \cite{GN06} (proved via a non-parsimonious reduction
to the permanent).

\begin{conj}\label{conj:bicirc}
Function \ts $b$ \ts restricted to bicircular matroids is concise.
\end{conj}

\smallskip

\subsection{Linear extensions of height two posets}\label{ss:finrem-LE-height-two}
In the context of Conjecture~\ref{conj:LE-height-two}, denote by \ts $e'$ \ts
the restriction of function \ts $e'$ \ts to posets height two.
The conjecture claims that $e'$ is almost complete, even though
the numerical evidence points in the opposite direction:
\begin{equation}\label{eq:h2-missing}
7, \. 9, \. 10, \.  11, \.  13, \.  15, \.  17, \.  18, \.  19, \.  21, \.  22, \.
23, \.  26, \.  27,  \. 28,  \. 29,  \. 31,  \. 32, \, \ldots \, \notin \. \cT_{e'}\..
\end{equation}
However, the number of height two posets on $n$ elements is asymptotically
\ts $\exp\Theta(n^2)$, thus overwhelming the numbers of linear extensions
(cf.\ Example~\ref{ex:sym-YT}).  In fact,
almost all posets have bounded height \cite{KR75}.  Additionally, \ts $\cT_{e'}$ \ts
can contain large primes, e.g.\ \ts $E_6 = 61 \in \cT_{e'}(6)$ \ts and
$$
E_{38} \, = \, 23489580527043108252017828576198947741 \, \in \. \cT_{e'}(38),
$$
see \cite[\href{http://oeis.org/A092838}{A092838}]{OEIS}.
Here \. $E_n := e(Z_n)$ \. is the \defn{Euler number} (also called \emph{secant number}),
defined as the number of linear extensions of the \defn{zigzag poset} \.
$x_1\prec x_2 \succ x_3 \prec x_4 \succ \ldots\.$, see e.g.\ \cite{Sta-Alt} and
\cite[\href{http://oeis.org/A000111}{A000111}]{OEIS}. Clearly, $Z_n$ has height two.
Thus we believe that positive evidence outweighs the negative, favoring the
conjecture.\footnote{Extensive computer computations by David Soukup (personal communication,
August 2023), show that \ts
$1593350922239999 \notin \ts \cT_{e'}$ \ts suggesting that \eqref{eq:h2-missing}
might not be an instance of the \emph{strong law of small numbers}~\cite{Guy88}. }

\smallskip

\subsection{Hamiltonian paths in tournaments}\label{ss:finrem-ham}
An orientation of all edges of a complete graph \ts $K_n$ \ts is called
a \defn{tournament}. Denote by \ts $h(T)$ \ts the number of (directed)
Hamiltonian paths in~$T$.  Following Szele (1943), there is large literature
on the maximal value of~$h$ over all tournaments with $n$ vertices,
see an overview in \cite[$\S$4.2]{KO12}. R\'edei (1934) proved that \ts
$h(T)$ \ts always odd, see e.g.\ \cite[$\S$9]{Moon} and \cite[$\S10.2$]{Berge}.
Curiously, it is not known if \ts $(h-1)/2$ \ts is almost complete:

\begin{conj}\label{conj:finrem-Ham-paths}
Function~$\ts h$ \ts takes all odd values except~$7$ and~$21$.
\end{conj}

This conjecture was made by the {\tt MathOverflow} user {\tt bof} and Gordon Royle
(March~2016).\footnote{\href{https://mathoverflow.net/q/232751}{mathoverflow.net/q/232751}}
Royle reported that the conjecture holds for all odd integers up to~80557 (ibid).

\noindent

\smallskip

\subsection{Degree sequences}\label{ss:finrem-deg-seq}
Let \. $\bd = (d_1,\ldots,d_n)\in \nn^n$, such that
\. $0\le d_1, \ldots, d_n \le n-1$.  Denote by \ts $c(\bd)$ \ts
the number of simple graphs on $n$ vertices \ts $V=\{v_1,\ldots,v_n\}$,
with \ts $\deg(v_i)=d_i$ \ts for all \ts $1\le i \le n$.
It was conjectured in \cite[$\S$13.5.4]{Pak-OPAC},
that computing~$\ts c(\cdot)$ \ts is $\SP$-complete.
 Note that the decision problem
\ts $\{c(\bd)>^?0\}$ \ts is in~$\P$ by the
\defng{Erd\H{o}s--Gallai theorem}, see \cite{EG61,SH91}.
We refer to \cite{Wor18} for a recent survey of the large literature on the
subject.

\begin{conj}\label{conj:finrem-deg-seq}
Function~$\ts c$ \ts is concise,
i.e.\ for every \ts $k>0$ \ts there exists \. $\bd = (d_1,\ldots,d_n)\in \nn^n$, such that
\ts $c(\bd) = k$ \ts and \. $n\le C(\log k)^c$, for some fixed \ts $C, c>0$.
\end{conj}

\noindent
Note that graphs with a given degree sequences can also be viewed as
\emph{binary  $(0/1)$ symmetric contingency tables}, making this conjecture a variation on
Conjecture~\ref{conj:CT}.

\smallskip

\subsection{Plane triangulations}\label{ss:finrem-triangulations}
Let \. $X=\{p_1,\ldots,p_n\}\in \qqq^2$ \. be the set of $n$ points in general positions, i.e.\
where no three points lie on a line.  Denote by \. $t(X)$ \. the number of triangulations of the
 convex hull of \ts $X$ \ts which contain all vertices in~$X$. It was conjectured in
\cite[Exc.~8.16]{DRS}, that computing \ts $t$ \ts is \ts $\SP$-complete (see also \cite{Epp20}
for a closely related problem).  It is known that \ts $t$ \ts
has exponential growth (see e.g.\ \cite[$\S$3.3]{DRS}):
$$
C_1 \cdot 2.3^n \. \le \. t(X) \. \le \. C_2 \cdot 43^n \quad \text{for some \ $C_1,C_2>0$.}
$$


\begin{conj}\label{conj:finrem-triangulations}
Function~$\ts t$ \ts is almost complete.
\end{conj}


\nin
By Proposition~\ref{p:concise-complete}, the conjecture implies
that \ts $t$~is concise.  One reason to believe the conjecture is
the exponential number of topological triangulations on~$n$ vertices,
given by \defng{Tutte's formula} (see e.g.\ \cite[Thm~3.3.5]{DRS}),
combined with \defng{F\'ary's theorem} \ts that all topological triangulations are
realizable as plane triangulations (see e.g.\ \cite[Thm~8.2]{PA95}). Thus, there
is no growth discrepancy as in the proof of Proposition~\ref{p:sym-SYT}.

\subsection{Trees in planar graphs}\label{ss:finrem-trees}
Let \ts $G=(V,E)$ \ts be a simple planar graph, and let \ts $tr(G)$ \ts
be the number of trees in~$G$ (of all sizes).  Trivially, the number of
trees in an empty graph with $n$ vertices is~$n$, so \ts $tr$ \ts is almost
complete.  On the other hand, Jerrum showed that  computing
\ts $tr$ \ts is $\SP$-complete \cite{Jer94}, but the proof
uses a non-parsimonious reduction to \ts {\sc \#2SAT}.
This suggests the following:

\begin{conj}\label{conj:finrem-trees}
Function~$\ts tr$ \ts is concise.
\end{conj}

\nin
Note that the number of planar graphs on $n$ vertices is asymptotically \ts
$C \ts  n^{-7/2} \ts \ga^n$ \ts  where \ts $\ga \approx 27.23$, see~\cite{GN09},
which is large enough to make the conjecture plausible.

\subsection{Critical group}\label{ss:finrem-CG}
In \cite[p.~19]{GK20}, Glass and Kaplan ask about the smallest number of vertices
\ts $n=n(\rH)$ \ts of a graph with a given \defng{critical group} \ts $\rH$
(also known as the \defng{sandpile group}), that can be defined by the
\emph{Smith normal form} \ts of the graph Laplacian.
Recall that for every graph $G$, the size of its critical group is the number
of spanning trees in~$G$ : \ts $|\rH_G|=\st(G)$.  Thus this problem refines
the problem in the Example~\ref{ex:var-st}.
In particular, Stong's theorem \cite{Stong} implies that \. $n(\rH)= o((\log k)^{3/2})$ \ts
for all \emph{square-free} \ts $k=|\rH|$, and suggests that the same bound holds for all~$\rH$.
We refer to \cite[Ch.~6]{CP18} and \cite[Ch.~4]{Kli19} for the background and further
references on the critical group.

\subsection{Determinants}\label{ss:finrem-det}
In their study of hypergraphs with positive discrepancy, Cherkashin and Petrov \cite{CP19}
considered the following problem.  Fix a parameter \ts $\de\in \nn$.  Denote by
\ts $\cT_\de(n)$ \ts the set of all possible determinants of \ts $n\times n$ \ts matrices
with entries in \ts $\{0,1,\ldots,\de\}$.  They show that
$$
\cT_4(n) \, \supset \, \big\{\ts 0,1,\ldots,c^n\ts \big\} \quad \text{for some} \ \ c > 1.
$$
This result shows that the determinant is a concise \ts $\GapP$ \ts function.
A stronger result is proved by Shah~\cite{Shah} in connection with random
\emph{binary $(0/1)$ matrices}:
$$
\cT_1(n) \, \supset \, \big\{\ts 0,1,\ldots, \ts \tfrac{\al \ts 2^n}{n}\.\big\}  \quad \text{for some} \ \ \al > 0,
$$
(cf.\ \cite[\href{http://oeis.org/A013588}{A013588}]{OEIS}).

In~\cite[$\S$4]{CP19}, the authors ask whether these results can be strengthened to
\begin{equation}\label{eq:CP-det}
\cT_\de(n) \, \supset \, \big\{\ts 0,1,\ldots, \ts c^{n\log n}\ts\big\}  \quad \text{for some} \ \ \de\ge 1 \ \text{and} \ c > 1.
\end{equation}
Denote \ts $\be_n := \max\cT_1(n)$.  Finding \ts $\be_n$ \ts is the famous
\defng{Hadamard maximal determinant problem}, see e.g.\  \cite{BC72,BEHO},
\cite[\href{http://oeis.org/A003432}{A003432}]{OEIS} and references therein.
It was shown by Hadamard (1893) that \ts $\be_n \le n^{n/2}$.  In a different direction,
it is known that \ts $\be_n >  C \ts e^{n\log n - cn}$ \ts for an infinite set of integer~$n$
and fixed $c,C>0$, see \cite[Cor.~27]{BEHO}.  This lends a (weak) partial evidence
towards~\eqref{eq:CP-det}.

Finally, recall that by the matrix-tree theorem, the number of spanning trees is a
determinant  of the Laplacian.  Thus, since the graphs in \cite{Stong} have degrees
at most six (cf.\ Example~\ref{ex:var-st}), Stong's result gives a
\eqref{eq:CP-det}--style inclusion for matrices with entries in \ts $\{-1,0,1,\ldots,6\}$, but
with a weaker upper limit \ts $c^{n^{2/3}}$.
By contrast, Azarija and \v{S}krekovski conjecture in \cite{AS13} that one can
take the upper bound to be \ts $e^{\om(n)}$ \ts for general simple graphs.
Of course, this bound cannot be achieved on graphs with bounded (average) degree,
since they have at most exponential number of spanning trees, see e.g.\ \cite{Gri76}.

\subsection{Reduced factorizations}\label{ss:finrem-reduced}
For a permutation \ts $\si \in S_n$, a \defn{reduced factorization} \ts is
a product \. $\si = (i_1,i_1+1)\cdots (i_\ell,i_\ell+1)$ \. of \ts
$\ell=\inv(\si)$ \ts adjacent transpositions. Denote by \ts $\red(\si)$ \ts
the number of reduced factorizations of~$\si$.  It was conjectured
in~\cite{DP18} that computing \ts $\red$ \ts is  $\SP$-complete.
Observe that function \ts $\red$ \ts is almost complete, since
$$
\red(2,1,n,3,4,\ldots,n-1) \, = \, n-2 \quad \, \text{for} \ \ \. n\ts \ge \ts 3.
$$

\begin{conj}\label{conj:finrem-reduced}
Function \ts $\red$ \ts is concise.
\end{conj}

Let \ts $\si\in S_n$ \ts be a \defn{$321$-avoiding} \ts permutation, i.e.\ suppose
that \ts $\pc_{321}(\si)=0$.  We have in this case: \ts $\red(\si)=\SYT(\la/\mu)$,
where \ts $\la/\mu$ \ts is a skew shape associated to~$\si$, see \cite[pp.~358-359]{BJS}.
Vice versa, for every skew shape \ts $\la/\mu$ \ts of size~$n$, there is a
$321$-avoiding permutation \ts $\si\in S_n$ \ts such that \ts $\red(\si)=\SYT(\la/\mu)$.
In other words, a positive answer to Question~\ref{q:SYT-skew} implies
Conjecture~\ref{conj:finrem-reduced}.
For further discussions and applications of this result, see e.g.\ \cite[$\S$2.6.6]{Man01}
and \cite[$\S$2.2, $\S$5.3]{MPP19}.


\subsection{Littlewood--Richardson coefficients}\label{ss:finrem-LR}
Recall that the Littlewood--Richardson coefficients is a complete function
(see Example~\ref{ex:LR-complete}), and that the maximal LR-coefficient
is exponential in~$n$, see \cite{PPY19}.  This suggest the following:

\begin{conj}\label{conj:LR-concise}
The Littlewood--Richardson coefficients \ts $c^\la_{\mu\nu}$ \ts is a concise function,
i.e.\ for every \ts $k>0$ \ts there exists three partitions \ts $\la, \mu, \nu$, such that
\ts $c^\la_{\mu\nu} = k$, \ts $|\la|=|\mu|+|\nu| =n$ \ts and \. $n\le C(\log k)^c$, for some fixed \ts $C, c>0$.
\end{conj}

\noindent
Using standard reductions, this conjecture follows from Conjecture~\ref{conj:CT}
that the numbers of contingency tables is a concise function, see e.g.\ \cite{PV}.
Let us mention that similarly to the domino tilings, the LR-coefficients have a
combinatorial interpretation as the number of \emph{Knutson--Tao puzzles} \ts
with boundary markings as shown in Figure~\ref{f:KT}, see e.g.~\cite{Knu22}.

\begin{figure}[hbt]
\begin{center}
	\includegraphics[height=2.2cm]{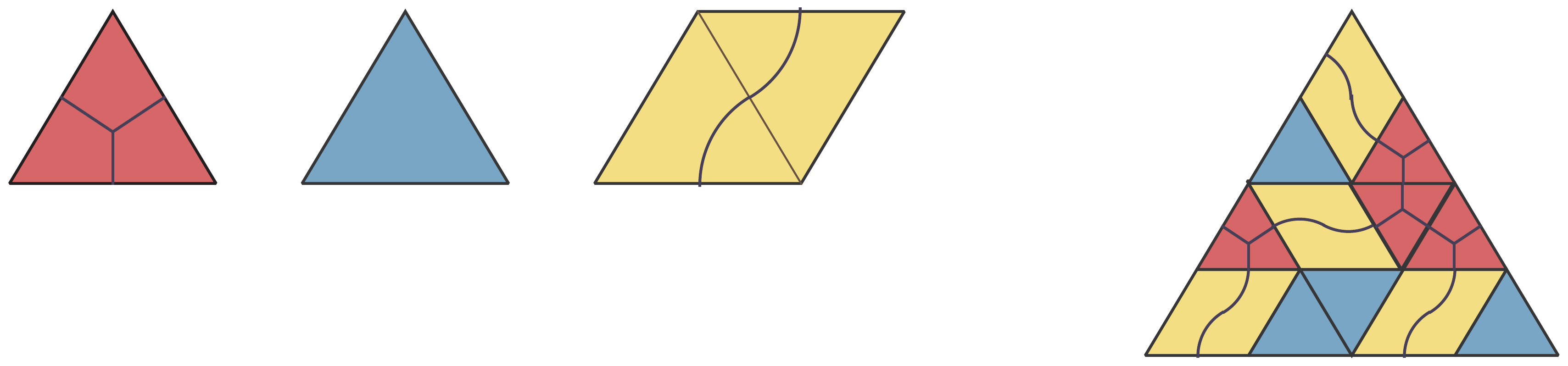}
\caption{Tiles of the Knutson--Tao puzzles (up to rotation, reflections are not allowed),
and an example of a puzzle. }
\label{f:KT}
\end{center}
\end{figure}


\subsection{Integer points}\label{ss:finrem-Knapsack}
Let \ts $P \in \rr^d$ \ts be a \defn{rational convex H-polytope}, i.e.\ defined
by inequalities over~$\qqq$, and let \ts $\zeta(P)$ \ts be the number of integer
points in~$P$.  One can ask if \ts $\zeta$ \ts is concise in full generality?
Here we are assuming that the input is in unary.  Note that \ts $\zeta$ \ts
is a counting function which generalizes the number of contingency tables
(Example~\ref{ex:CT-complete}) and LR-coefficients
(Example~\ref{ex:LR-complete}), but we don't know if either of these is concise
(see Conjectures~\ref{conj:CT} and~\ref{conj:LR-concise}).

There is a sneaky way to establish that  \ts $\zeta$ \ts is concise, by
noting that in a special case this becomes one of the known \ts
$\SP$-complete problems in unary, with a parsimonious reduction
from the \ts {\sc 0/1 PERMANENT} (see e.g.\ \cite[Thm~1.2]{DO04}) \ts
or the \ts {\sc \#3SAT} (see e.g.\ \cite[$\S$5.1]{IMW17} and references therein).
Recall that both of these problems have concise counting functions.
In both cases, Proposition~\ref{ss:SAT-reduction} implies that \ts $\zeta$ \ts
is concise.  We omit the details.

\subsection{Domino tilings of simply connected regions}
Recall that our proof of Theorem~\ref{t:domino} uses regions that are not
simply connected.  Note also that Nadeau's construction in Remark~\ref{r:Nadeau}
is simply connected.

\begin{conj} \label{conj:domino-simply-connected}
Theorem~\ref{t:domino} holds for simply connected regions.
\end{conj}

If true, the proof would have to be much more elaborate than our proof for
general regions.  Indeed, note that domino tilings of simply connected regions
have additional properties given by the \defng{height functions} \ts
(see e.g.\ \cite{PST,Rem,Thu}).

Additionally, the \emph{generalized Temperley's bijection} \cite{KPW}
relates the number of domino tilings in simply connected regions with
the number of spanning trees in certain grid graphs.  Thus,
Conjecture~\ref{conj:domino-simply-connected} can be reformulated
to say that a certain restriction of \ts $\st$ \ts is
concise, see Example~\ref{ex:var-st}.  This gives another reason
to think that the conjecture might be difficult.
%
%

On the other hand, it follows from the approach \cite{CS18,Sch19} and 
\cite{CP-CF} that Conjecture~\ref{conj:domino-simply-connected} follows from
\emph{Zaremba's conjecture}.  Moreover, for the \emph{snake regions} 
defined in \cite{CS18} this approach gives the following analogue of \eqref{eq:KS};  
we omit the details.

\smallskip

\begin{thm}  \label{t:domino-sc}
There is \ts $c>1$, such that \. 
$\cT_{\text{\rm snake}}(n) \supseteq \big\{0,1,\ldots,c^{n/(\log n)}\big\}$,
for all \ts $n\ge 1$.  
\end{thm}

\subsection{Tilings with small tiles}\label{ss:finrem-tile}
The literature on tilings in the plane is much too large to be reviewed here,
but let us mention a few relevant highlight.
In \cite{MR01}, the authors proved $\NP$-completeness
for the decision problem and $\SP$-completeness for the counting problem of
tilings of general regions for a small set with just two tiles
(up to rotation).  In \cite{BNRR}, \ts $\NP$-completeness was proved for
tilings with \emph{bars} \ts $k \times 1$ \ts and \ts $1\times k$,
where \ts $k\ge 3$.  This is especially remarkable
since for the same set of tiles, tileability of simply connected regions is
in~$\P$, see~\cite{KK93}.  Finally, in~\cite{PY-sc}, a finite set of rectangle
tiles was given, with an $\NP$-complete decision problem and $\SP$-complete
corresponding counting problem for simply connected regions.

\subsection{The induction two-step}\label{ss:finrem-groups}
The inductive arguments in the paper can be used to extend some results of concise
functions.  This phenomenon can be seen, for example, in
Corollary~\ref{c:gen} which arises naturally from the proof of Theorem~\ref{t:domino}.
Similarly, for relatively prime \ts $(a,b)$, the analogue of the corollary holds
for linear extensions of width two posets (Example~\ref{ex:LE-restricted}),
and the number of matchings of trees (Example~\ref{ex:var-match}).  This is a
corollary of properties of the Stern--Brocot tree (see references in~\cite{KS21}).
Another notable example of this phenomenon is a beautiful result in~\cite{HKNS},
that for every two groups \ts $\Ga_0$ \ts and \ts $\Ga_1$, there is a graph \ts
$G=(V,E)$ \ts and an edge \ts $e\in E$, such that \ts Aut$(G)=\Ga_0$ \ts
and \ts  Aut$(G-e)=\Ga_1$\ts.

\subsection{Concise functions}\label{ss:finrem-open-problems}
As we mentioned earlier, the notion of a concise function is closely related to \ts
$\SP$-completeness. Indeed, for every parsimonious reduction \ts $f\mapsto g$,
if \ts $f$ \ts is concise, then so is~$g$.  Below we give an example
of a $\SP$-complete function that is \emph{not} \ts concise.
Note that if either Conjecture~\ref{conj:LE-height-two} or
Conjecture~\ref{conj:bicirc} is false, this would also give
such examples.

On the other hand, note that there are several interesting examples of concise
functions that are in~$\FP$, such as the number of spanning trees in simple graphs
and the number of linear extensions of width two posets (Examples~\ref{ex:var-st}
and~\ref{ex:LE-restricted}).  This suggests that being concise is an interesting
property in its own right, independently of complexity considerations.

\subsection{Coincidence problems}\label{ss:finrem-coincidence}
Note that none of the natural $\SP$-complete counting functions~$f$ that
we consider have the corresponding coincidence problems \ts {\sc C}$_f$ \ts
in~$\ts\PH$ (unless $\PH$ collapses).  Theorems~\ref{t:main-first} and~\ref{t:main-combined} prove
this in many cases, and some interesting examples remain open
(see above).
The same applies for the verification problem \ts {\sc V}$_f\ts$.
Note that if \ts $f$ \ts is recognizable, then these two problems are $\PH$-equivalent,
i.e.\ \. $\text{\sc C}_f \in \PH \ \Longleftrightarrow \ \text{\sc V}_f \in \PH$ \.
by the argument in the proof of Proposition~\ref{p:rec}.

The following elegant example by Greta Panova shows that  $\SP$-complete functions
can have easy coincidence problems.\footnote{Personal communication (August, 2023)}
In notation of~$\S$\ref{ss:perm-gen}, let \ts $A=(a_{ij})\in \cM_n$ \ts where
$\cM_n$ \ts is the set of $n \times n$ \ts matrices with entries in \ts $\{0,1\}$.
Denote \ts $g(A) := a_{11} + 2 a_{12} + 4 a_{13} + \ldots + 2^{n^2-1} a_{nn} + 2^{n^2}$ \ts
and let \ts $f: \sqcup \ts\cM_n \to \nn$ \ts be defined by \ts $f(A):= \per(A) + 2^{n^2} g(A)$.
Observe that \ts $\per(A)\le n! < 2^{n^2}$.
Since \ts $f(A) = \per(A) \mod 2^{n^2}$, we conclude that computing $\ts f \ts $ is
\ts $\SP$-complete.  On the other hand, we have \ts $f(A) = f(A')$ \ts
if and only if \ts $A=A'$, so \ts {\sc C}$_f \in \P$.

Curiously, function \ts $f$ \ts is \emph{not} \ts almost complete since
the first \ts $(n^2+1)$ \ts bits determine the last  \ts $n^2$ \ts bits, but is
concise since \ts $n^2 = O\big(\log f(A)\big)$ \ts for all \ts $A\in \cM_n$\ts.
Function \ts $f$ \ts is not recognizable, however, since
the corresponding restricted verification problem is exactly \ts {\sc V}$_{\text{PER}}$ \ts
which is not in~$\PH$ (unless $\PH$ collapses).   Now take \ts $\cM_n':=\cM_n \times [n]$.
Define \ts $f'(A,\ell):= f(A)$ \ts if \ts $\ell=0$, and \ts $f'(A,\ell):= \ell-1$ \ts if \ts
$\ell > 0$.  Then function \ts $f'$ \ts is complete, and thus \emph{not} \ts concise.
Of course, computing \ts $f'$ \ts remains \ts $\SP$-complete.

This leaves us with more questions than answers.  Is \ts {\sc C}$_{\text{PER}} \in \coNP$-hard?
Is it \ts $\CEP$-complete under Turing reductions?  What about
other natural coincidence problems that we discuss in this paper?
Does it make sense to consider a complexity class of \ts $\SP$-complete
recognizable functions as in Proposition~\ref{p:rec}?

\subsection{Equality conditions}\label{ss:finrem-equality-conditions}
The equality conditions for many combinatorial inequalities, have been studied
widely in the area, especially in order theory.  We refer to \cite{Gra,Win86}
for dated surveys, to \cite{Pak-OPAC} for a recent overview.
Computationally, the equality conditions are decision
problems in \ts $\CEP$.  This paper grew out of our efforts to understand the
Alexandrov--Fenchel inequality for mixed volumes \cite{CP+}.


\vskip.8cm


\subsection*{Acknowledgements}
We are grateful to Karim Adiprasito, Sasha Barvinok, Richard Brualdi,
Jes\'{u}s De~Loera, Darij Grinberg, Christian Ikenmeyer, Mark Jerrum,
Nathan Kaplan, Greg Kuperberg, Alejandro Morales, Steven Noble, Greta Panova,
Fedya Petrov, Colleen Robichaux, Rikhav Shah, David Soukup,
Richard Stanley and Bridget Tenner, for useful discussions
and helpful remarks.  Special thanks to Noah Kravitz,
Ashwin Sah and Richard Stong for their help explaining \cite{KS21} and~\cite{Stong}.

This paper was finished when both authors were visiting the American
Institute of Mathematics at their new location in Pasadena, CA.
We are grateful to AIM for the hospitality.  The first
author was partially supported by the Simons Foundation.
Both authors were partially supported by the~NSF.




\end{document}